\documentclass{article}
\usepackage[utf8]{inputenc} \usepackage[T1]{fontenc}    \usepackage[english]{babel}

\usepackage{amsfonts}
\usepackage{nicefrac}       \usepackage{microtype}      \usepackage{lmodern}
\usepackage{amssymb,amsmath,amsthm}
\usepackage{stmaryrd}
\SetSymbolFont{stmry}{bold}{U}{stmry}{m}{n}
\usepackage{bbm,bm}
\usepackage{latexsym}
\usepackage{xcolor}

\usepackage{enumerate}
\usepackage{verbatim}
\usepackage{booktabs}       

\usepackage{url}            \usepackage[colorlinks=true]{hyperref}
\usepackage[numbers,sort&compress,square,comma]{natbib}
\usepackage[capitalise,noabbrev]{cleveref}

\usepackage{parskip}
\usepackage{geometry}
\usepackage[short]{optidef}
\usepackage{algorithm} 
\usepackage{algorithmic}  
\usepackage[algo2e]{algorithm2e} 
\usepackage{thmtools}

\declaretheorem{theorem}

\declaretheorem{lemma}
\declaretheorem{proposition}

\declaretheoremstyle[qed=$\square$]{definitionwithend}

\declaretheorem[style=definitionwithend]{remark}

\crefname{assumption}{Assumption}{Assumptions}

\definecolor{gold}{rgb}{0.85,0.65,0}

\newcommand{\by}{\times}
 
\newcommand{\norm}[1]{\ensuremath{\left\lVert #1 \right\rVert}}
\newcommand{\ip}[1]{\ensuremath{\left\langle #1 \right\rangle}}
\newcommand{\grad}{\ensuremath{\nabla}}

\newcommand{\set}[1]{\left\{#1\right\}}

\newcommand{\mb}{\mathbf}

\def\R{{\mathbb{R}}}

\def\cF{{\cal F}}

\def\cH{{\cal H}}
\def\cI{{\cal I}}

\DeclareMathOperator*{\argmin}{arg\,min}

\DeclareMathOperator{\range}{range}

\DeclareMathOperator{\spann}{span}

\newcommand{\framedheader}[3]{
  \framebox[\textwidth]{
    \vbox{
      \vspace{2mm}
      \hbox to \textwidth {\hspace{1em}\today #1 \hfill #2\hspace{1em}}
      \vspace{4mm}
      \hbox to \textwidth {\hfill \Large{#3} \hfill}
      \vspace{2mm}
    }
  }
  \vspace*{4mm}
} \usepackage{soul}

\newcommand{\btau}{\pmb{\tau}}
\newcommand{\bo}{\textbf{1}}
\newcommand{\bh}{\textbf{h}}

\newcommand{\bq}{\textbf{q}}

\usepackage{tikz}

\begin{document}

\title{Beyond Minimax Optimality: A Subgame Perfect Gradient Method}
\author{
    Benjamin Grimmer\footnote{Johns Hopkins University, Department of Applied Mathematics and Statistics, \texttt{grimmer@jhu.edu}}
    \and
    Kevin Shu\footnote{California Institute of Technology, Computational and Mathematical Sciences, \texttt{kshu@caltech.edu}}
    \and
    Alex L.\ Wang\footnote{Purdue University, Daniels School of Business, \texttt{wang5984@purdue.edu}}
}
	\date{\today}
	\maketitle

\abstract{%
	The study of convex optimization has historically been concerned with worst-case convergence rates.
	The development of the Optimized Gradient Method (OGM), due to \cite{drori2012PerformanceOF,Kim2016optimal}, marked a major milestone in this study, as OGM achieves the optimal worst-case convergence rate among all first-order methods for unconstrained smooth convex optimization.
	In order to examine the possibility of obtaining stronger convergence guarantees, we will consider algorithms with \emph{dynamic} convergence rates, which may improve as additional first-order information is revealed.
	Our main contribution is the development of an algorithm, the Subgame Perfect Gradient Method (SPGM), which refines OGM to make use of the full history of first-order information.
	We show that SPGM is \emph{dynamically optimal}, in the sense that in each iteration, no other algorithm can offer a strictly better convergence rate on all functions which agree with the observed first-order information up to that iteration.
	We formalize this notion of dynamic optimality using the game-theoretic notion of a subgame perfect equilibrium.
	We conclude our study with preliminary numerical experiments showing that SPGM strongly outperforms OGM.
}

\section{Introduction}
This paper considers black-box first-order methods for smooth convex minimization.
Let $f : \R^d \rightarrow \R$ be an $L$-smooth convex function with a minimizer $x_\star \in \R^d$. We wish to construct algorithms that effectively minimize the normalized suboptimality:
\[
\frac{f(x_N) - f(x_\star)}{L\|x_0 - x_\star\|^2/2},
\]
where $x_N$ is the final iterate of the method and $x_0$ is the starting iterate of the method. Throughout all norms denote the Euclidean norm associated with inner product $\langle \cdot,\cdot\rangle$.

Classically, the design of first-order methods has focused on algorithms with good (or in some cases, optimal) worst-case guarantees~\cite{nemirovskij1983problem}.
In \cite{drori2012PerformanceOF,Kim2016optimal}, a momentum based first-order method known as the Optimized Gradient Method (OGM) was exhibited, and this method was shown in \cite{drori2017exact} to achieve the best possible worst-case performance on the class of $L$-smooth convex functions. This is often referred to as \emph{minimax optimality}.
To be precise, it was shown in \cite{Kim2016optimal} that for any $L$-smooth convex function $f$, OGM guarantees
\[
\frac{f(x_N) - f(x_\star)}{L\|x_0 - x_\star\|^2/2} \le \frac{1}{\tau_{0,N}},
\]
where we define $\tau_{0,N}$ in \eqref{eq:ogm-recurrence}.\footnote{We state the \eqref{eq:ogm-recurrence} 
	in terms of a sequence $\tau_{0,0},\tau_{0,1},\dots,\tau_{0,N}$. This differs from the original presentation in \cite[Equation 6.13]{Kim2016optimal} that involves a sequence $\theta_0,\dots,\theta_N$. Equivalence between the two recurrences follows from the change of variables: 
	$\tau_{0,i} = 2\theta_i^2$ for $i=0,\dots,N-1$ and $\tau_{0,N}=\theta_N^2$.}
For context, the quantity $\tau_{0,N}$ grows asymptotically like $\tau_{0,N} \approx \frac{N^2}{2}$.
It was later shown in \cite{drori2017exact} that as long as $d \geq N+2$, there exists an $L$-smooth convex function $f_{\textup{hard}}$ so that for any gradient span first-order method, 
\[
\frac{f_{\textup{hard}}(x_N) - f_{\textup{hard}}(x_\star)}{L\|x_0 - x_\star\|^2/2} \ge \frac{1}{\tau_{0,N}}.
\]
Recall, a first-order method is a gradient span method if at each iteration, $x_i \in x_0 + \spann \{\nabla f(x_0), \dots, \nabla f(x_{i-1})\}$.

\begin{figure}[tb]
	\centering
	
	\begin{tikzpicture}[scale=2.1]
		\def\xs{1, -0.6180339887498949, 0.45588678010286676, -0.3636639571190878, 0.30350121938992136
		};
		
		\draw[->] (-1.25, 0) -- (1.25, 0) node[right] {$x$};
		\draw[->] (0, -0.25) -- (0, 1);
		
		\node[right] at (1.25, 0.5 * 1.25 * 1.25) {$f(x)$};
		\draw[domain=-1.25:1.25, smooth, variable=\x]  plot ({\x}, {\x * \x / 2});
		
		\foreach \x [count=\n from 0] in \xs {
			\node at (\x, \x * \x / 2) [circle,fill,inner sep=1.5pt,label=below:$x_{\n}$]{};
		}
		
		\foreach \x [count=\n from 0, remember=\x as \prev] in \xs {
			\ifnum \n > 0
			\draw[dashed] (\prev, {0.5 * \prev * \prev}) -- (\x, {0.5 * \x * \x});
			\fi
		}
	\end{tikzpicture}\quad
	\begin{tikzpicture}[scale=2.1]
		\def\xs{1, -0.6180339887498949, 0
		};
		
		\draw[->] (-1.25, 0) -- (1.25, 0) node[right] {$x$};
		\draw[->] (0, -0.25) -- (0, 1);
		
		\node[right] at (1.25, 0.5 * 1.25 * 1.25) {$f(x)$};
		\draw[domain=-1.25:1.25, smooth, variable=\x]  plot ({\x}, {\x * \x / 2});
		
		\foreach \x [count=\n from 0] in \xs {
			\node at (\x, \x * \x / 2) [circle,fill,inner sep=1.5pt,label=below:$x_{\n}$]{};
		}
		
		\foreach \x [count=\n from 0, remember=\x as \prev] in \xs {
			\ifnum \n > 0
			\draw[dashed] (\prev, {0.5 * \prev * \prev}) -- (\x, {0.5 * \x * \x});
			\fi
		}
	\end{tikzpicture}
	\caption{Left: The first five iterates of OGM on $f(x) =Lx^2/2$ with $x_0 = 1$. OGM produces $x_4 \approx 0.304$. Right: The first three iterates of SPGM on $f(x) =Lx^2/2$ with $x_0 = 1$. After seeing the history $\cH=\set{(x_0,f_0,g_0),(x_1,f_1,g_1)}$, SPGM determines that $0$ is a minimizer for \emph{any} $L$-smooth convex function agreeing with $\cH$. In effect, the history $\cH$ completely determines the function $f$ on the interval $[x_0,x_1]$.}
	\label{fig:example_suboptimal}
\end{figure}
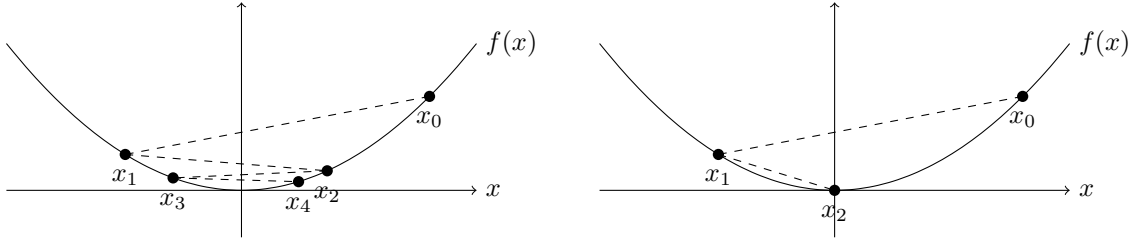

Crucially, note that OGM's convergence rate, determined by $\tau_{0,N}$, does not adapt as a function of the first-order information actually encountered by the algorithm. The step sizes and momentum terms used in each OGM iteration similarly do not depend on the first-order information encountered by the algorithm.
This suggests the possibility that both the analysis and the updates in OGM may be improved by taking more of the learned first-order information into account.
As a simple example, consider the quadratic function $f(x)= \frac{L}{2}x^2$ with initial iterate $x_0=1$ (see \cref{fig:example_suboptimal}).
It is known~\cite[Theorem 5.1]{kim2017convergence} that OGM's behavior on this function satisfies $\frac{f(x_N)- f(x_\star)}{L\norm{x_0-x_\star}^2/2} = \frac{1}{\tau_{0,N}}$, i.e., $\frac{L}{2}x^2$ is in fact a worst-case instance for OGM.
However, we will see that an algorithm which remembers first-order queries can, in fact, identify a minimizer of this function given just two first-order queries.
We would like to formalize the advantage that one can obtain by making use of the history of first-order information given to the algorithm, and also develop algorithms which make full use of that advantage.

Our first main contribution is a new algorithm, the Subgame Perfect Gradient Method (SPGM), which is naturally viewed as a refinement of OGM, that makes use of observed first-order information in order to achieve stronger dynamic convergence rates. That is, if the first-order information revealed to the algorithm happens to be informative, then the algorithm will be able to leverage that information to both dynamically update its remaining update rules and its convergence guarantees. These dynamic convergence guarantees can be orders of magnitude stronger than the static guarantees of traditional methods like OGM or adaptive methods, discussed in Section~\ref{subsec:relatedWork}. Section~\ref{sec:numerical} provides preliminary numerics demonstrating such gains.

Our second main contribution is to show that SPGM is in fact \emph{dynamically optimal} in a certain sense. To do this, we will need to introduce a notion of dynamic optimality for first order algorithms, which we refer to as being \emph{subgame perfect}.
This optimality notion requires an algorithm to provide optimal guarantees at every stage of its execution, as a function of the observed first-order information up to that stage.
The term ``subgame perfect'' references the game-theoretic notion of a ``subgame perfect equilibrium'', which will be defined shortly.

We present the game-theoretic framework next and then state precisely SPGM's performance guarantees in \cref{thm:subgamePerfect}.

\subsection{A game-theoretic framework for analysis of first-order methods}
\label{subsec:game}
We consider a zero-sum game\footnote{
	The view of optimization algorithms as playing a game against an adversary is not new and dates back at least to \cite{nemirovskij1983problem}.
}, which we refer to as the \emph{minimization game}.
This game is played between two players, Alice and Bob, where
Alice represents the first-order method and Bob represents the first-order oracle.
We will say that a set of triples $\{(x_i, f_i, g_i)\}_{i\in\cI}$ indexed by some index set $\cI$ is \emph{interpolable} if there exists an $L$-smooth convex function $f$ so that $f(x_i) = f_i$ and $\nabla f(x_i) = g_i$ for all $i\in\cI$.

For a fixed $d \ge 0$, the game proceeds in $N+1$ rounds, indexed by $n = 0, \dots, N$.
In round $n$, Alice chooses a point $x_n \in x_0 + \spann\set{g_0,\dots,g_{n-1}}$ representing a query point, and Bob responds with a pair $(f_n, g_n)$ with $f_n \in \R$ and $g_n \in \R^d$ 
so that $\{(x_i, f_i, g_i)\}_{i=0}^n$ is interpolable.
At the end of round $N$, Bob additionally specifies a triple $(x_\star, f_\star, g_\star) \in \R^{d}\times \R\times\R^d$ with $g_\star=0$ so that $\{(x_i, f_i, g_i)\}_{i\in[0,N]\cup\set{\star}}$ is interpolable.
Alice's payoff is the reciprocal of the normalized suboptimality:
\[
\frac{L\norm{x_0 - x_\star}^2/2}{f(x_N) - f(x_\star)}.
\]
We interpret this payoff to be $+\infty$ whenever the denominator is zero.

The subgames where Alice has a decision to make can be fully specified by the first-order history $\cH = \set{(x_i,f_i,g_i)}_{i=0}^{n-1}$, which is guaranteed to be interpolable.
A deterministic strategy for Alice in the minimization game is a mapping from each first-order history to a query point $x_{n}\in x_0 + \spann(\set{g_0,\dots,g_{{n-1}}})$.

A strategy for Bob can be defined similarly as a map sending any first-order history $\cH$ and query point $x_n$ to first-order data $(f_n, g_n)$, as well as a map from each first-order history $\set{(x_i,f_i,g_i)}_{i=0}^{N}$ to a choice of minimizer $x_\star$ and minimum $f_\star$.
The payoff of a pair of strategies (from Alice's perspective) is Alice's final payoff if both players make moves in accordance to their respective strategies.

\subsection{Equilibrium notions in the minimization game}

We say that a pair of strategies is in \emph{Nash Equilibrium} if neither player can increase their payoff by changing strategies (while the other player maintains the same strategy).
In the minimization game, one strategy for Alice would be to choose query points consistent with OGM (see \cref{alg:OGM}).
The OGM guarantee states
\begin{equation*}
	f(x_N) - f(x_\star) \leq \frac{L}{2\tau_{0,N}}\norm{x_0-x_\star}^2.
\end{equation*}
That is, Alice is guaranteed a payoff of at least $\tau_{0,N}$ regardless of Bob's play.

One strategy for Bob would be to report the function and gradient values from the $L$-smooth function $f_\textup{hard}$.
This strategy is available as long as $d \geq N+2$.
By \cite[Theorem 3]{drori2017exact}, Bob can guarantee that
\begin{equation*}
	f_N - f_\star \geq \frac{L}{2\tau_{0,N}}\norm{x_0-x_\star}^2
\end{equation*}
regardless of Alice's play, i.e., Alice is guaranteed a payoff of at most $\tau_{0,N}$.
Thus, these two strategies form a Nash Equilibrium; neither player can increase their payoff by deviating their strategies.

While the notion of Nash Equilibria is natural in \emph{static} games, finer notions exist for \emph{sequential} games.
A pair of strategies is said to be in \emph{Subgame Perfect Equilibrium} if for every subgame, the restriction of the two strategies to that subgame are in Nash Equilibrium.
This notion is stronger than that of Nash Equilibrium (see \cite{osborne2004introduction} for details on game theory).
To illustrate these ideas,
consider the situation where Alice and Bob respectively play the OGM strategy and first-order history from the function $f(x) = \frac{L}{2}x^2$ for rounds $n=0$ and $n=1$. At the beginning of round $n=2$, Alice will be in the subgame corresponding to $\cH=\set{(x_0,f_0,g_0), (x_1,f_1,g_1)}$. From this point onward, 
if Alice and Bob continue with their strategies, Alice will receive a payoff of exactly $\tau_{0,N}$. On the other hand, Alice may deviate from the OGM strategy in this subgame to guarantee a strictly larger payoff.
Indeed, as depicted in \cref{fig:example_suboptimal},
$x_1-g_1/L$ is guaranteed to be a minimizer for \emph{any} function $f$ interpolating $\cH$.
Thus, by playing the strategy $x_1-g_1/L = x_2 = \dots = x_N$, Alice can guarantee herself an infinite payoff in this subgame, regardless of Bob's strategy.
This shows that OGM does \emph{not} belong to a subgame perfect equilibrium.

This notion is practically important when the function we are trying to optimize is not adversarially chosen.
We may model the non-adversarial nature of typical functions as Bob playing this game ``suboptimally'', i.e., playing a strategy that is not part of a Nash Equilibrium for the subgame.
Thus, a strategy for Alice which is part of a Subgame Perfect Equilibrium can be regarded as an algorithm which capitalizes on suboptimal play to the largest extent possible, while still accounting for the possibility that Bob may begin to play adversarially in the future.

The following theorem states the performance guarantees of SPGM. The upper bound guarantee is proved in Lemma~\ref{lem:n_step_guarantee} and the lower bound construction is presented and proved in Section~\ref{sec:lower_bounds}.
\begin{theorem}
	\label{thm:subgamePerfect}
	For any $0 \leq n \leq N$ and any set of first-order history $\cH=\set{(x_i,f_i,g_i)}_{i\in[0,n-1]}$ generated by SPGM, the output $x_N$ of running SPGM for $N-n+1$ additional iterations satisfies
	\begin{equation*}
		f(x_N) - f(x_\star) \leq \frac{L}{2\tau_{n,N}}\norm{x_0-x_\star}^2\leq \frac{L}{2\tau_{0,N}}\norm{x_0-x_\star}^2,
	\end{equation*}
	where $\tau_{n,N}$ is defined in~\eqref{eq:tau_ni} and depends on the history $\cH$ generated by SPGM.
	Additionally, if $d\geq N+2$, then there exists an $L$-smooth convex function $f:\R^d\to\R$ interpolating $\cH$ such that the output $x_N$ from \emph{any} method satisfying
	\begin{equation}
		\label{eq:gradient_span_condition}
		x_i \in x_0 + \spann\left(\set{g_0,\dots, g_{i-1}}\right),\qquad\forall i\in[n,N]
	\end{equation}
	satisfies
	\begin{equation*}
		f_N-f_\star \geq \frac{L}{2\tau_{n,N}}\norm{x_0-x_\star}^2.
	\end{equation*}
\end{theorem}

This result establishes a dynamic minimax optimality of SPGM with respect to a sequence of subproblem classes. For a given history $\cH=\set{(x_i,f_i,g_i)}_{i\in[0,n-1]}$, let $\mathcal{A}$ denote the family of all gradient span methods and $\mathcal{A}^\cH$ denote the subset of methods that reproduce $x_0,\dots x_{n-1}$ given oracle responses $\cH$ and let $\mathcal{F}$ denote the set of all $L$-smooth convex problems and $\mathcal{F}^\cH$ denote the subset of problems interpolating $\cH$. Then the above guarantee establishes that
$$ \min_{A\in \mathcal{A}^\cH} \max_{f\in \mathcal{F}^\cH} \frac{f(x_N) - f(x_\star)}{\frac{L}{2}\|x_0-x_\star\|^2_2} = \frac{1}{\tau_{n,N}}$$
and is attained by SPGM.

We make explicit some consequences of this statement. Taking $n=0$, we see that Alice can guarantee a payoff of at least $\tau_{0,N}$ after observing $0$ oracle responses. This is the same guarantee provided by the OGM.
Furthermore, if Bob plays optimally throughout the course of the entire game, then no strategy for Alice can guarantee a strictly larger payoff.
In round $n=1$, after observing $\set{(x_0,f_0,g_0)}$, Alice updates her guaranteed payoff to $\tau_{1,N}\geq \tau_{0,N}$.
If Bob had played suboptimally in round $n=0$, then Alice can capitalize on Bob's mistake and guarantee a strictly larger payoff in this subgame, so $\tau_{1,N}>\tau_{0,N}$.
Furthermore, if Bob plays optimally in rounds $n=1,\dots,N$ (i.e., the remainder of the game), then no strategy for Alice can guarantee a strictly larger payoff in this subgame.
Naturally, it will hold that $\tau_{0,N} \leq \tau_{1,N}\leq \dots\leq \tau_{N,N}$, where $\tau_{0,N}$ is OGM's guarantee and $\tau_{N,N}$ is SPGM's guarantee after $N$ rounds.

While $\tau_{N,N}=\tau_{0,N}$ can occur on adversarial problems, we numerically see $\tau_{N,N}$ is often many orders of magnitude larger. For example, Figure~\ref{fig:representative-experiment} in our numerics shows two sample problems where SPGM's guarantee is approximately $10^5$ and $10^7$ times larger than OGM's after only a few hundred iterations.

\begin{remark}
	Technically, SPGM is not a ``strategy'' for Alice as SPGM does not assign an action to certain game states, namely those that cannot occur as a result of running SPGM on some first-order oracle.
	This is not practically relevant, however, because we are primarily concerned with the design of algorithms and may ignore suboptimal play by Alice.
	It can be shown via backward induction that there is a subgame perfect equilibrium $(A, B)$ so that $A$ extends SPGM in the sense that $A$ plays the same action as SPGM in all cases where SPGM is defined.
\end{remark}

SPGM can be derived constructively through the Performance Estimation Problem (PEP) framework \cite{drori2012PerformanceOF,taylor2017interpolation}. Specifically, SPGM is an optimized version of OGM where the quantities $\set{(f_i,x_i,g_i)}_{i\in[0,n-1]}$, which are treated as \emph{variables} in the PEP framework before the start of the minimization game, become constants at round $n$.
Surprisingly, the computation required to implement the SPGM amounts to solving a single \emph{convex} optimization problem in each iteration. Thus, the SPGM can be implemented in a numerically stable manner with some storage and computational overhead.
Both the storage and computational overhead can be controlled by considering a \emph{limited-memory version} of SPGM,
analogous to the limited-memory version of the classic BFGS algorithm known as L-BFGS. 
In \cref{sec:lSPGM}, we give a precise description of this limited-memory variant of SPGM and show that it has a storage overhead of $O(dk)$ and a computational overhead of at most $O(k^{3.5} + dk)$ time per iteration, where $k$ is a user-specified parameter.

\subsection{Related work}\label{subsec:relatedWork}
Below, we summarize classical and recent results in the design of optimal gradient methods and first-order methods using a bundle of past first-order information.

\paragraph{Classical Convergence Theory for Gradient Methods.}
Gradient descent, which dates back at least to \cite{cauchy1847methode}, is perhaps the simplest algorithm for smooth convex optimization.
In this setting, gradient descent with stepsize $1/L$, i.e., the algorithm $x_{n} = x_{n-1} - \frac{1}{L}g_{n-1}$ can be shown to achieve suboptimality $f_N - f_\star \leq \frac{L\norm{x_0 - x_\star}^2}{2N}$~\cite{d2021acceleration}.
In a seminal result of~\cite{Nesterov1983}, a method was constructed with the guarantee
$f_N - f_\star \leq \frac{2L\norm{x_0 - x_\star}^2}{N^2}$.
Almost matching lower bounds were given by~\cite{nemirovsky1992information,nemirovskij1983problem} showing that $f_N - f_\star \leq \frac{2L\norm{x_0 - x_\star}^2}{N^2}$ is optimal up to absolute constants.

\paragraph{Recent Optimized Gradient Methods.}
Since the early 2010s, there has been an explosion of work seeking optimal methods in large part due to the Performance Estimation Problem (PEP) framework initiated by \cite{drori2012PerformanceOF,taylor2017interpolation}. PEPs enable one to exactly compute via semidefinite programming the worst-case performance of a given first-order method over structured families of problem instances (e.g., all $L$-smooth convex problems).
Leveraging this exact tool, a significant list of \emph{minimax optimal} or conjectured minimax optimal first-order methods has been discovered.
In the game-theoretic framework of \cref{subsec:game}, minimax optimal corresponds to the best possible payoff for Alice that she can guarantee when playing against an optimal adversary.

The Optimized Gradient Method for smooth convex optimization was first observed numerically in~\cite{drori2012PerformanceOF} and proved analytically in~\cite{Kim2016optimal}. It was proved to be minimax optimal in~\cite{drori2017exact}. Adaptive versions of this method have been designed using linesearching~\cite{Drori2018EfficientFM,TaylorBach2019a} as well as restarting~\cite{Kim2018Restarts} techniques. An optimized first-order method for strongly convex minimization was derived in~\cite{VanScoy2018,Cyrus2018,Park2021Factor}, culminating in the development of the minimax optimal Information-Theoretic Exact Method of~\cite{taylor2022optimal}. Optimized methods able to be applied to various constrained settings or settings with proximal terms have been developed as well~\cite{taylor2017smooth,Kim2018FISTA,Kim2021PPM,Barre2022,jang2024composite}.
The minimax optimal constant stepsize gradient descent scheme for smooth convex or strongly convex optimization was first observed numerically in~\cite{drori2012PerformanceOF} and proved in~\cite{rotaru2024exact,kim2024proof}.
Conjectured minimax optimal long stepsize gradient descent schemes with partial acceleration were first observed numerically in~\cite{gupta2023branch} with upper bounds proved analytically in~\cite{altschuler2023accelerationPartI,altschuler2023accelerationPartII,grimmer2024accelerated,grimmer2024composing,zhang2024acceleratedgradientdescentconcatenation,zhang2024anytime}. 
In conjunction with the development of these optimized performance upper bounds, techniques have been developed for producing exactly matching performance \emph{lower bounds}~\cite{drori2017exact,Drori2021OnTO}.
Our analysis in \cref{sec:lower_bounds} relies heavily on the \emph{zero-chain} constructions of~\cite{Drori2021OnTO}.

These prior optimized methods and matching lower bounds aim to establish minimax optimal performance against the whole class of $L$-smooth convex problems. In contrast, our theory targets a stronger \emph{dynamic} minimax optimality. At every iteration, SPGM dynamically adjusts to additionally attain the minimax optimal performance against the subclass of $L$-smooth convex problems that agree with the history of first-order information revealed so far.

\paragraph{Bundle Methods in Nonsmooth Optimization.}
SPGM's optimized updates require maintaining a memory of all past first-order oracle responses $\{(x_i,f_i,g_i)\}_{i=0}^{n-1}$ received, where $f_i=f(x_i),g_i= \nabla f(x_i)$. Such a collection of first-order queries in the nonsmooth optimization literature with subgradients $g_i\in\partial f(x_i)$ is often referred to as a bundle. ``Bundle methods'' utilizing such a  collection have found strong practical success, first being proposed by~\cite{Lemarechal1975,Wolfe1975} with a large literature of supporting convergence theory~\cite{Kiwiel1983,Hiriart1993,Ruszczynski2006,Warren2010,Du2017,DiazGrimmer2023,Liang2021} giving worst-case performance bounds. Most relevant to our development here is the Kelley cutting plane-Like Method (KLM) developed by~\cite{drori2016Kelley}. Therein, the authors designed updates based on a series of relaxations and reformulations to produce an iterate with an optimized worst-case performance as a function of the bundle $\{(x_i,f_i,g_i)\}_{i=0}^{n-1}$. In doing so, they prove KLM has a dynamic sequence of upper bounds, never worse than the minimax optimal convergence rate.
Forthcoming work establishes that KLM is a subgame perfect method in the setting of $M$-Lipschitz nonsmooth convex optimization.

\paragraph{Gradient Methods with Memory.}
The incorporation of memory into accelerated gradient methods for smooth convex optimization has been accomplished by the level bundle method of Lan~\cite{Lan2015} and into proximal bundle methods by Fersztand and Sun~\cite{fersztand2025}. Both of these works show accelerated worst-case guarantees of $O(1/N^2)$ but guarantee no further dynamic improvements from their memory. Such dynamic improvements were the focus of the recent series of works on gradient methods with memories of~\cite{Nesterov2022memory,Florea2022exact,FloreaNesterov2025OptimalLowerBound}. At each iteration, these gradient methods with memory approximately solve a nonlinear subproblem to dynamically improve the inductive step in an associated estimate sequence-style proof. Further,~\cite{FloreaNesterov2025OptimalLowerBound} constructs dynamic lower bounds from observed gradients, providing an ``optimized'' lower bounding model. The spirit of dynamically improving an algorithm's induction at runtime and dynamically constructing hard instances is shared with SPGM. These works prove their methods with memory achieve at least the optimal worst case rate of $O(1/N^2)$. Establishing whether their gradient methods with memory achieve or nearly achieve subgame perfection is an interesting future direction. These methods have been extended to strongly convex, composite settings and to allow backtracking~\cite{florea2021composite,florea2024adaptivefirstordermethodsenhanced}. Determining whether SPGM can also be extended to such settings is left open as well.

Quasi-Newton methods are also a gradient method with memory. Such algorithms only rely on first-order oracle evaluations, but use a history of gradients to construct an approximate Hessian via BFGS-like updates~\cite{broyden1970,fletcher1970,goldfarb1970,shanno1970} and take approximate Newton steps. Limited memory BFGS methods often provide strong practical performance. However, non-asymptotic guarantees for quasi-Newton methods are limited. Preliminary numerics presented in Section~\ref{sec:numerical} show that SPGM is nearly as performant as L-BFGS while providing strong, subgame perfect, non-asymptotic guarantees.

\paragraph{Uniform/Universal methods.}
A significant research vein in recent years has targeted the design of \emph{uniform/universal} methods~\cite{nesterov2015universal,Lan2015,malitsky2019adaptive,li2025simple}. These are algorithms which may be run without prior knowledge of problem class or associated problem constants. For example, \cite{nesterov2015universal}'s Universal Fast Gradient Method (UFGM) when applied to minimize any convex function with $(\beta,p)$-H\"older continuous gradient is guaranteed to achieve an objective gap convergence guarantee of $O((\beta \|x_0-x_\star\|^{1+p}/\epsilon)^{2/(1+3p)})$. For any $\beta>0, p\in[0,1]$, UFGM's rate is big-O optimal for the associated class of H\"older smooth, convex problems.

Generally one can view universal algorithms as providing optimal guarantees against a range of target problem classes, e.g., $(\beta,p)$-H\"older smooth convex minimization for any $\beta,p$. In contrast, methods like OGM provide minimax optimal guarantees against exactly one class of problems, $L$-smooth convex minimization for fixed $L$. Our notion of subgame perfect methods is a distinct approach: Subgame perfect methods provide optimal guarantees against a dynamic sequence of problem classes, namely $L$-smooth convex functions agreeing with the observed problem history up to iteration $n$ for each $n=0,\dots, N$.

\subsection{Outline}
We provide some background information on PEP and OGM in \cref{sec:preliminaries}. \cref{sec:upper_bounds} introduces the SPGM formally and proves its convergence guarantee (\cref{thm:spgm_performance}).
\cref{sec:lower_bounds} constructs matching worst-case functions agreeing with given history (see \cref{prop:interpolating,prop:zero-chain}).
The resulting dynamic matching lower bounds together with \cref{lem:n_step_guarantee} constitute a proof for \cref{thm:subgamePerfect}.
In \cref{sec:lSPGM}, we describe a limited-memory variant of SPGM and provide its guarantees.
We conclude with preliminary numerical experiments in \cref{sec:numerical} comparing SPGM with gradient descent, OGM, Limited-Memory BFGS~\cite{liu1989limited}, and UFGM.

\section{Preliminaries}
\label{sec:preliminaries}

\subsection{Smooth convex minimization and performance estimation}
\label{subsec:pep_prelims}
This section summarizes background material on the Performance Estimation Problem (PEP) methodology, initiated by~\cite{drori2012PerformanceOF,taylor2017interpolation}, in the setting of unconstrained smooth convex optimization.

Our goal is to find an optimal solution to the unconstrained minimization problem
\begin{equation*}
	\min_{x\in\R^d}f(x),
\end{equation*}
where $f$ is known to be convex and $L$-smooth, and is assumed to have a minimizer. 

Recall, a convex function $f:\R^d\to\R$ is said to be $L$ smooth if it is differentiable and for all $x,y\in\R^d$,
\begin{equation*}
	\norm{\grad f(x)-\grad f(y)}\leq L\norm{x-y}.
\end{equation*}


Instead of thinking about the function $f$ itself, it will be useful to think about the collection of first-order data that a first-order method encounters:
$\set{(x_i,f_i,g_i)}_{i\in\cI}$.
Here, $\cI$ is some index set and we use the shorthand $f_i=f(x_i)$ and $g_i=\grad f(x_i)$.
The following lemma gives a characterization of 
when a set of first-order data is \emph{interpolable} by an $L$-smooth convex function. 
\begin{lemma}[{\cite[Corollary 1]{taylor2017interpolation}}]
	\label{lem:interpolation}
	Given a set of first-order data $\set{(x_i,f_i,g_i)}_{i\in\cI}$, there exists an $L$-smooth convex function $f$ satisfying
	\begin{equation*}
		f(x_i )= f_i\qquad\text{and}\qquad \grad f(x_i)= g_i,\qquad\forall i\in\cI
	\end{equation*}
	if and only if the quantity
	\begin{equation*}
		Q_{i,j}\coloneqq f_i - f_j - \ip{g_j,x_i-x_j}- \frac{1}{2L}\norm{g_i-g_j}^2
	\end{equation*}
	is nonnegative for all $i,j\in\cI$.
\end{lemma}

It will be helpful to introduce a slightly different parameterization of the first-order data: for a given $x\in\R^d$, we define 
\begin{equation}
	\label{eq:plus_defs}
	f(x)^+ = f(x) - \frac{1}{2L}\norm{\grad f(x)}^2\qquad\text{and}\qquad
	x^+ \coloneqq x - \frac{1}{L}\grad f(x).
\end{equation}
Intuitively,
$x^+$ is the result of a single step of gradient descent from $x$ and $f(x)^+$ is the ``natural'' upper bound on $f(x^+)$ from $L$-smoothness:
\begin{equation*}
	f(x^+) \leq f(x) + \ip{\grad f(x), x^+ - x} + \frac{L}{2}\norm{x^+ - x}^2 = f(x)^+.
\end{equation*}
For convenience, let $f_i^+ = f(x_i)^+$ and let $x_i^+ = (x_i)^+$.
These quantities will appear prominently in both the OGM and SPGM. Note also that with these definitions, the quantity $Q_{i,j}$ can be rewritten as
\begin{equation*}
	Q_{i,j} = f_i^+ - f_j^+ - \langle g_j, x_i^+ - x_j^+\rangle,
\end{equation*}
which mirrors the first-order inequalities associated to a convex function.

\subsection{Optimized Gradient Method (OGM) and Auxiliary Sequences}
\label{subsec:ogm}

Our Subgame Perfect Gradient Method generalizes the existing Optimized Gradient Method (OGM). Below, we introduce OGM and (our interpretation of) an analysis for OGM first given by~\cite{TaylorBach2019a} and presented clearly in~\cite[Theorem 4.4]{d2021acceleration} and~\cite{Park2021Factor}.

OGM generates iterates $x_0,x_1,\dots,x_N$ as well as an auxiliary sequence $z_1,\dots,z_{N+1}$ and a sequence of parameters $\tau_{0,0},\dots,\tau_{0,N}>0$, which roughly measure the rate of convergence of the algorithm in each iteration.
Define the following ``increment'' function
\begin{equation}
	\label{eq:delta}
	\delta_n(\tau) \coloneqq
	\begin{cases}
		1 +\sqrt{1 + 2\tau} &\text{if }n\leq N-1\\
		\frac{1+\sqrt{1+4\tau}}{2} &\text{else}
	\end{cases},
\end{equation}
which will be used by OGM to increment the value of $\tau_{0,n}$.

The OGM algorithm is formally described in \cref{alg:OGM}.
Note that at the time that $z_{n+1}= z_n -\frac{\delta_n(\tau_{0,n-1})}{L} g_n$ is defined, the quantity $g_n$ is still unknown.
What we mean by this definition is that $z_{n+1}$ is a \emph{formal expression} in $g_n$ that becomes a concrete vector in $\R^d$ upon learning $g_n$ at the beginning of iteration $n+1$.

\begin{algorithm}
	\caption{OGM}
	\label{alg:OGM}
	Given $L$-smooth convex function $f$, initial iterate $x_0$, iteration budget $N$
	\begin{itemize}
		\item Define $z_1 = x_0 - \frac{2}{L}g_0$, $\tau_{0,0}=2$. 
		\item For $n = 1,\dots,N$
		\begin{enumerate}
			\item Reveal $g_{n-1} = \grad f(x_{n-1})$
			\item Define
			\begin{align}
				\label{eq:ogm-recurrence}
				\tag{OGM recurrence}
				\tau_{0,n} & \coloneqq \tau_{0,n-1} + \delta_n(\tau_{0,n-1}),\\
				x_n &\coloneqq \frac{\tau_{0,n-1}}{\tau_{0,n}}x_{n-1}^+ + \left(\frac{\delta_n(\tau_{0,n-1})}{\tau_{0,n}}\right) z_n, \nonumber\\
				z_{n+1}&\coloneqq z_n - \frac{\delta_n(\tau_{0,n-1})}{L} g_n.
				\nonumber
			\end{align}
		\end{enumerate}
	\end{itemize}
\end{algorithm}

\begin{remark}
	As stated, the OGM (\cref{alg:OGM}) requires the iteration budget $N$ as part of its input. However, this $N$ is never used in the definition of the updates to $x_n$ or $z_{n+1}$ until the final iteration $n=N$. That is, OGM only needs to know the iteration budget $N$ \emph{a single iteration in advance}. In this sense, OGM is \emph{almost any-time}: we can run OGM without an iteration budget fixed in advance, as long as we allow it to perform one additional iteration after we ask it to terminate. This same property will hold for the subgame perfect gradient method (\cref{alg:SPGM}).

\end{remark}
\begin{remark}
	We choose the notation $\tau_{0,n}$ instead of the more traditional single index notation $\tau_{n}$ to emphasize the fact that $\tau_{0,n}$ is known at the beginning of iteration 0, i.e., it does not depend on the oracle responses.
	By contrast, when we analyze SPGM, we will introduce sequences $\tau_{n,n},\dots,\tau_{n,N}$ that depend on the first $n$ oracle responses $\set{(x_i,f_i,g_i)}_{i=0}^{n-1}$.
\end{remark}

Both our analysis of OGM and our analysis of SPGM will involve the notion of an \emph{auxiliary} vector, which we define for a fixed $L$-smooth convex function $f:\R^d \rightarrow \R$ with minimizer $x_\star \in \R^d$ and an arbitrary $x_0 \in \R^d$.
We say that $z \in \R^d$ is an \emph{auxiliary vector for $x \in \R^d$ with rate $\tau \in \R$} if
\[
f(x) - f_\star + \frac{L}{2\tau}\norm{z - x_\star}^2\leq \frac{L}{2\tau}\norm{x_0 - x_\star}^2.
\]
As a simple consequence, if $x$ has some auxiliary vector with rate $\tau$, then
\[
f(x) - f_\star \leq \frac{L}{2\tau}\norm{x_0 - x_\star}^2.
\]

Intuitively, this definition parameterizes (by $\tau$) a tradeoff between $f(x)-f_\star$ and $\norm{z-x_\star}$. For example, if $f(x)- f_\star$ achieves its worst-case upper bound $\frac{L}{2\tau}\norm{x_0-x_\star}^2$, then necessarily $z=x_\star$. Similarly, if $\norm{z-x_\star}$ achieves its worst-case upper bound $\norm{x_0-x_\star}$, then necessarily $f(x)=f_\star$.

We will further define the \emph{pre-rate} of an \emph{auxiliary vector}. We say that  $z \in \R^d$ is an \emph{auxiliary vector for $x \in \R^d$ with pre-rate $\tau \in \R$} if
\[
f(x)^+ - f_\star + \frac{L}{2\tau}\norm{z - x_\star}^2\leq \frac{L}{2\tau}\norm{x_0 - x_\star}^2.
\]
The difference between the definitions of pre-rate and rate is that the quantity $f(x)$ is replaced with $f(x)^+= f(x) - \frac{1}{2L}\norm{\nabla f(x)}^2$ in the definition of the pre-rate.

A full analysis of OGM is captured by the following two lemmas, which we will reuse in our analysis of SPGM in \cref{sec:upper_bounds}.

\begin{lemma}
	\label{lem:OGM_initialization}
	$z_1 = x_0- \frac{2}{L}g_0$ is an auxiliary vector for $x_0$ with pre-rate $2$.
\end{lemma}

\begin{lemma}
	\label{lem:OGM_induction}
	Suppose $z\in\R^d$ is an auxiliary vector for $x\in\R^d$ with pre-rate $\tau\in\R$.
	Let $\delta = 1 + \sqrt{1+2\tau}$ and set
	\begin{align*}
		\tau' &= \tau + \delta,\\
		x' &= \frac{\tau}{\tau'}x^+ + \frac{\delta}{\tau'} z,\\
		z' &= z - \frac{\delta}{L} \nabla f(x').
	\end{align*}
	Then, $z'$ is an auxiliary vector for $x'$ with pre-rate $\tau'$.
	If, in the definition above, $\delta$ is replaced by $\delta = \frac{1+\sqrt{1+4\tau}}{2}$, then $z'$ is an auxiliary vector for $x'$ with rate $\tau'$.
\end{lemma}
These lemmas are implicitly proved in~\cite{TaylorBach2019a}. We include their proofs in \cref{sec:deferred} for completeness.

Combining these two lemmas, we see that for each iteration $n\in[0,N-1]$ of OGM, $z_{n+1}$ is an auxiliary vector for $x_n$ with pre-rate $\tau_{0,n}$ and $z_{N+1}$ is an auxiliary vector for $x_N$ with rate $\tau_{0,N}$. We conclude that the output of OGM satisfies
\begin{equation*}
	f(x_N) - f_\star \leq \frac{L}{2\tau_{0,N}}\norm{x_0-x_\star}^2.
\end{equation*}

\section{Subgame perfect strategies for the Minimization Game}
\label{sec:upper_bounds}
This section introduces the \emph{Subgame Perfect Gradient Method} (SPGM), and gives an upper bound on its final normalized suboptimality.
SPGM is naturally thought of as an augmentation of OGM, which was defined in \cref{subsec:ogm}.
In particular, we will use the notion of an auxiliary sequence, which was used in our explanation of OGM.

In \cref{subsec:SPGM}, we will begin stating the SPGM algorithm and state a bound on the final normalized suboptimality of the algorithm. We will complete the description of SPGM in \cref{subsec:subproblem}. There, we introduce and analyze \eqref{eq:optimize_lambdas}, the subroutine which is the key driver of the improvement in performance from OGM to SPGM.
In \cref{subsec:primal_feasibility}, we will clarify some aspects of \eqref{eq:optimize_lambdas} such as its feasibility and what happens when the subproblem becomes unbounded.
Finally, in \cref{subsec:intermediate}, we will note that the bounds on the convergence rate of the SPGM algorithm will naturally improve as new gradients are revealed, and offer some bounds on the final suboptimality of the algorithm which can be computed at the $n^{th}$ iteration (but not earlier).

\subsection{The Subgame Perfect Gradient Method (SPGM)}
\label{subsec:SPGM}

SPGM is formally stated in \cref{alg:SPGM}. As in OGM, we will maintain two sequences $x_0,\dots,x_N$ and $z_1,\dots,z_{N+1}$.
The last step of each iteration is also identical to OGM, but with an intermediate set of values $\tau_{n-1/2}, x_{n-1/2},$ and $z_{n+1/2}$ replacing $\tau_{n-1,n-1},x_{n-1},$ and $z_n$ respectively. We denote these values with a half increment in each subscript to emphasize these quantities are calculated between the construction of $(\tau_{n-1,n-1},x_{n-1},z_{n})$ and $(\tau_{n,n},x_{n},z_{n+1})$.
These intermediate quantities are defined by \eqref{eq:optimize_lambdas}, which we define in the next subsection.

\begin{algorithm}[H]
	\caption{SPGM}
	\label{alg:SPGM}
	Given $L$-smooth convex function $f$, initial iterate $x_0$, iteration budget $N$
	\begin{itemize}
		\item Define $\tau_{0,0} = 2$ and $z_1 = x_0 - \frac{2}{L}g_0$
		\item For $n = 1,\dots,N$
		\begin{enumerate}
			\item Reveal $(f_{n-1},g_{n-1}) = (f(x_{n-1}),\grad f(x_{n-1}))$ and update first-order information set $\cH\coloneqq \set{(x_i,f_i,g_i)}_{i={0}}^{n-1}$
			\item Define 
			$x_{n-1/2}= x_m$ where $m\in\argmin_{i\in[0,n-1]}f_i^+$ and let $\tau_{n-1/2}$ be the optimal value of \eqref{eq:optimize_lambdas}.
			\item If \eqref{eq:optimize_lambdas} is unbounded, output $x_{n-1/2}^+$. Otherwise, define $z_{n+1/2}$ by \eqref{eq:x_half_def}.
			\item Define
			\begin{align*}
				\tau_{n,n} & \coloneqq \tau_{n-1/2} +\delta_n(\tau_{n-1/2}),\\
				x_n & \coloneqq \frac{\tau_{n-1/2}}{\tau_{n,n}}x_{n-1/2}^+ + \frac{\delta_n(\tau_{n-1/2})}{\tau_{n,n}} z_{n+1/2},\\
				z_{n+1} & \coloneqq z_{n+1/2} - \frac{\delta_n(\tau_{n-1/2})}{L} g_n,
			\end{align*}\label{eq:SPGMDefinitions}
			where $\delta_n(\cdot)$ is defined in \eqref{eq:delta}.
		\end{enumerate}
	\end{itemize}
\end{algorithm}

In our analysis, we will see that $z_{n+1/2}$ is an auxiliary vector for $x_{n-1/2}$ with pre-rate $\tau_{n-1/2}$.
The intuition for introducing an intermediate set of values is that OGM defines the vector $z_{n}$ without seeing the oracle response $(f_n,g_n)$.
Once this oracle response is revealed, it may be possible to find an alternative auxiliary vector inside the span of the seen gradients which has a larger pre-rate. This would ultimately lead to a faster rate of convergence.
The subroutine \eqref{eq:optimize_lambdas} can be viewed as finding an auxiliary vector $z_{n+1/2}$ for $x_{n-1/2}$ which has as large a pre-rate as possible.
Indeed, our analysis implies that $\tau_{n-1/2}$ is the maximum pre-rate that can be guaranteed based on the history $\cH$. This is ultimately what leads to SPGM being subgame perfect.

The analysis of the performance of SPGM will build on the analysis of OGM. In \cref{lem:half_update_pc}, we will formally argue that $z_{n+1/2}$ is an auxiliary vector for $x_{n-1/2}$ with pre-rate $\tau_{n-1/2}$ for each $n \in [1,N]$. In \cref{lem:strict_feasibility_ogm} and \cref{lem:unbounded}, we will address feasibility and unboundedness of SPGM's subproblem. With these results in hand, we will then prove convergence guarantees on SPGM in \cref{thm:spgm_performance}.

\subsection{The SPGM Subproblem}
\label{subsec:subproblem}

We now explain the main algorithmic innovation in the SPGM iteration that produces the intermediate quantities $\tau_{n-1/2}$, $x_{n-1/2}$ and $z_{n+1/2}$.

Intuitively, we would like to directly maximize the value of $\tau$ subject to the existence of some $z \in \R^d$ so that $z$ is an auxiliary vector to $x_{n-1/2}$ with pre-rate $\tau$. That is, we want to maximize $\tau$ with the constraint that there exists $z \in \R^d$ so that 
\[
f^+_{n-1/2} - f_\star + \frac{L}{2\tau}\norm{z - x_\star}^2\leq \frac{L}{2\tau}\norm{x_0 - x_\star}^2.
\]
However, this inequality involves the unknown quantities $f_\star$ and $x_\star$, and thus, we cannot directly optimize $\tau$. We must instead infer bounds on the values of $f_\star$ and $x_\star$ from the observed first-order history.

We will make use of the following two types of bounds\footnote{Nonnegativity of $H_i$ comes from the inductive assumption that $z_{i+1}$ is an auxiliary vector for $x_i$ with pre-rate $\tau_{i,i}$. Nonnegativity of $Q_{\star,i}$ comes from \cref{lem:interpolation}.}:
\begin{align*}
	H_i&\coloneqq \tau_{i,i}\left(f_\star - f_i^+\right) + \frac{L}{2}\norm{x_0-x_\star}^2 - \frac{L}{2}\norm{z_{i+1}- x_\star}^2\geq 0\qquad\text{for }i\in[0,n-1]\\
	Q_{\star,i} &= f_\star - f_i^+ - \ip{g_i, x_\star - x_i^+} \geq 0 \qquad\text{for }i\in[0,n-1].
\end{align*}
We will ultimately maximize the value $\tau$ subject to the existence of $z \in \R^d$ \emph{and} a conic combination of the $H_i$'s and $Q_{\star, i}$'s that yield an inequality implying that $z$ is an auxiliary vector for $x_{n-1/2}$ with pre-rate $\tau$.

We introduce notation to compactly express these conic combinations:
let $Z\in\R^{d\times n}$ have columns indexed by $[0,n-1]$ with $i$th column $z_{i+1}-x_0$ and let $G\in\R^{d\times n}$ have columns indexed by $[0,n-1]$ with $i$th column $\frac{1}{L}g_i$.

That is, 
\begin{equation}
	Z = \begin{pmatrix}
		\vert & \vert & & \vert\\
		z_1-x_0 & z_2 - x_0 & \dots & z_{n}-x_0\\
		\vert & \vert & & \vert
	\end{pmatrix} \quad\text{and}\quad
	G = \frac{1}{L} \begin{pmatrix}
		\vert & \vert & & \vert\\
		g_0 & g_1 & \dots & g_{n-1}\\
		\vert & \vert & & \vert
	\end{pmatrix}.
\end{equation} 
Define the vectors $\btau,\bh,\bq\in\R^n$ indexed by $[0,n-1]$:
\begin{align*}
	\btau_i &\coloneqq \tau_{i,i}\\
	\bh_i &\coloneqq \tau_{i,i} f^+_i - \frac{L}{2}\norm{x_0}^2 + \frac{L}{2}\norm{z_{i+1}}^2\\
	\bq_i &\coloneqq f_i^+ -\ip{g_i,x_i^+}.
\end{align*}

The above notation is defined so that for any $\mu, \lambda_\star \in \R^n$,
\begin{align*}
	\sum_{i=0}^{n-1} \mu_i H_i + \sum_{i=0}^{n-1}\lambda_{\star,i}Q_{\star,i} &= \left(\ip{\btau,\mu}+\ip{\mb 1, \lambda_\star}\right)f_\star 
	+ L\ip{Z\mu -  G\lambda_\star,x_\star} - \ip{\bh, \mu} - \ip{\bq,\lambda_\star}.
\end{align*}

We may now state \eqref{eq:optimize_lambdas}:
\begin{equation}
	\label{eq:optimize_lambdas}
	\sup_{\mu,\lambda_\star\in\R^n}\set{ \tau :\, \begin{array}{l}
			\frac{L}{2}\left(\norm{\tilde{z}}^2 - \norm{x_0}^2\right)\leq 
			\ip{\mu,\bh-f^+_{n-1/2}\btau} + \ip{\lambda_\star, \bq - f^+_{n-1/2}\mb 1}\\
			\tilde{z} \ =\  x_0+ Z\mu - G\lambda_\star\\
			\tau = \ip{\btau,\mu} + \ip{\bo,\lambda_\star} \\
			\mu,\lambda_\star\geq 0
	\end{array}}
	\tag{SPGM-Subproblem}
\end{equation}
where $f_{n-1/2}^+= f_m^+$ with $m\in\argmin_{i\in[0,n-1]}f_i^+$ in agreement with our choice of  $x_{n-1/2}= x_m$ in SPGM.

Let $\tau_{n-1/2}$ denote the optimal value of \eqref{eq:optimize_lambdas}, possibly infinite.
Finally, if \eqref{eq:optimize_lambdas} has a bounded optimal solution, we denote the optimizing value of $\tilde{z}$ by
\begin{equation}
	\label{eq:x_half_def}
	z_{n+1/2} = x_0 + Z\mu - G\lambda_\star.
\end{equation}
In the following lemma, we establish that any feasible solution to this subproblem provides a valid auxiliary vector $\tilde{z}$. This result is the core step towards establishing our convergence guarantees on SPGM in Theorem~\ref{thm:spgm_performance}.

\begin{lemma}
	\label{lem:half_update_pc}
	Suppose that $H_i\geq 0$ for $i\in[0,n-1]$, $(\mu, \lambda_\star)$ are feasible for \eqref{eq:optimize_lambdas} with objective value $\tau$ and let $\tilde{z} = x_0 + Z\mu - G\lambda_\star$.
	Explicitly, we are supposing that $(\mu, \lambda_\star)$ are nonnegative and satisfy 
	\[
	\frac{L}{2}\left(\norm{\tilde{z}}^2 - \norm{x_0}^2\right)\leq 
	\ip{\mu,\bh-f^+_{n-1/2}\btau} + \ip{\lambda_\star, \bq - f^+_{n-1/2}\mb 1},
	\]
	and that $\tau = \ip{\btau,\mu} + \ip{\bo,\lambda_\star}$.
	
	Then, $\tilde{z}$ is an auxiliary vector for $x_{n-1/2}$ with pre-rate $\tau$.
\end{lemma}
\begin{proof}
	Let $(\mu, \lambda_\star)$ be feasible for \eqref{eq:optimize_lambdas}
	and let $\tilde{z} = x_0 + Z\mu - G\lambda_\star$. Fix $\tau = \ip{\btau,\mu} + \ip{\bo,\lambda_\star}$. We want to show that
	\[
	f_{n-1/2}^+ - f_\star + \frac{L}{2\tau}\norm{\tilde{z} - x_\star}^2\leq \frac{L}{2\tau}\norm{x_0 - x_\star}^2.
	\]
	
	This is equivalent to showing that the following quantity is nonnegative:
	\[
	H_{n-1/2}\coloneqq\tau \left(f_\star - f_{n-1/2}^+ \right)+
	\frac{L}{2}\norm{x_0 - x_\star}^2 - \frac{L}{2}\norm{\tilde{z} - x_\star}^2.
	\]
	
	We claim that 
	\begin{equation*}
		H_{n-1/2} \geq \sum_{i=0}^{n-1}\mu_i H_i + \sum_{i=0}^{n-1}\lambda_{\star,i}Q_{\star,i}.
	\end{equation*}
	This will imply that $H_{n-1/2} \ge 0$ since each $H_i$ is nonnegative by induction, $Q_{\star, i}$ is nonnegative by \cref{lem:interpolation}, and each $\mu_i, \lambda_{\star,i}$ is nonnegative by feasibility in \eqref{eq:optimize_lambdas}.
	
	We expand the right-hand side as follows:
	\begin{align*}
		&\sum_{i=0}^{n-1}\mu_i H_i + \sum_{i=0}^{n-1}\lambda_{\star,i}Q_{\star,i}\\
		&\qquad=(\langle \btau, \mu \rangle + \langle \bo, \lambda_{\star} \rangle) f_\star +L \langle Z\mu -G\lambda_\star, x_\star\rangle - \langle \bh, \mu\rangle -\langle \bq, \lambda_\star \rangle&& \\
		&\qquad=\tau (f_\star - f_{n-1/2}^+)  +L \langle \tilde{z} - x_0, x_\star\rangle - \langle \bh - f_{n-1/2}^+\btau, \mu\rangle -\langle \bq - f_{n-1/2}^+\bo, \lambda_\star \rangle&&\\
		&\qquad\leq \tau (f_\star - f_{n-1/2}^+)  +L \langle \tilde{z} - x_0, x_\star\rangle - \frac{L}{2}\left(\|\tilde{z}\|^2 - \|x_0\|^2\right)&&\\
		&\qquad=\tau\left(f_\star - f_{n-1/2}^+\right) + \frac{L}{2}\norm{x_0-x_\star}^2 - \frac{L}{2}\norm{\tilde{z}-x_\star}^2&&\\
		&\qquad=H_{n-1/2}.
	\end{align*}
	Here, the third line substitutes the definition
	$\tau = \ip{\btau, \mu} + \ip{\mb 1,\lambda_\star}$ and $\tilde{z} = x_0 + Z\mu - G\lambda_\star$, and the fourth line uses $\langle \bh - f_{n-1/2}^+\btau, \mu\rangle + \langle \bq - f_{n-1/2}^+\bo, \lambda_\star \rangle \geq \frac{L}{2}(\|\tilde{z}\|^2 - \|x_0\|^2)$, which holds at any feasible solution $(\mu,\lambda_\star)$ to \eqref{eq:optimize_lambdas}.
\end{proof}

\subsection{Feasibility and Potential Unboundedness of the Subproblem}
\label{subsec:primal_feasibility}

We now clarify various aspects of SPGM, such as the feasibility of \eqref{eq:optimize_lambdas} and consequences of early termination in step 3 of SPGM
\begin{lemma}
	\label{lem:strict_feasibility_ogm}
	The problem \eqref{eq:optimize_lambdas} is feasible and its optimal value is at least $\tau_{n-1,n-1}$.
	Strong duality holds between \eqref{eq:optimize_lambdas} and its dual.
\end{lemma}
\begin{proof}
	Let $e_i$ denote the basic vector in the $(i+1)$-th coordinate direction. Consider the solution $\mu = \alpha e_{n-1}$ and $\lambda_{\star} = 0$ parameterized by $\alpha \geq 0$,
	and set
	\[
	\tilde{z} = \alpha (z_{n}-x_0) + x_0.
	\]
	Below we check the feasibility of $(\alpha e_{n-1}, 0)$ in \eqref{eq:optimize_lambdas}. Note that
	\[
	\ip{\mu,\bh-f^+_{n-1/2}\btau} + \ip{\lambda_\star, \bq - f^+_{n-1/2}\mb 1} = \alpha \left(\tau_{n-1,n-1} \left(f^+_{n-1} - f^+_{n-1/2}\right)- \frac{L}{2}\norm{x_0}^2 + \frac{L}{2}\norm{z_{n}}^2\right),
	\]
	so $(\alpha e_{n-1}, 0)$ is feasible if and only if the following is nonnegative:
	\begin{align*}
		&\ip{\mu,\bh-f^+_{n-1/2}\btau} + \ip{\lambda_\star, \bq - f^+_{n-1/2}\mb 1} - \frac{L}{2}\left(\norm{\tilde{z}}^2 - \norm{x_0}^2\right)\\
		&\qquad = 
		\alpha \left(\tau_{n-1,n-1} \left(f^+_{n-1} - f^+_{n-1/2}\right)- \frac{L}{2}\norm{x_0}^2 + \frac{L}{2}\norm{z_{n}}^2\right)\\
		&\qquad\qquad - \frac{L}{2}\left(
		\alpha^2\norm{z_n -x_0}^2 + 2\alpha\ip{z_n - x_0, x_0}\right)\\
		&\qquad = \alpha\tau_{n-1,n-1} \left(f^+_{n-1} - f^+_{n-1/2}\right)+ \frac{L(\alpha-\alpha^2)}{2}\norm{x_0-z_{n}}^2.
	\end{align*}
	By assumption, $f_{n-1}^+ \geq f_{n-1/2}^+$. Thus by substituting $\alpha=1$, we have that $(\mu,\lambda_\star)=(e_{n-1},0)$ is feasible in \eqref{eq:optimize_lambdas} with objective value $\tau_{n-1,n-1}$.
	
	By convention, if \eqref{eq:optimize_lambdas} is unbounded, then strong duality holds. We assume \eqref{eq:optimize_lambdas} is bounded. In this case, we claim $z_n\neq x_0$. Indeed, if $z_n=x_0$, then $(\alpha e_{n-1},0)$ is feasible for all $\alpha\to\infty$ with an objective value $\alpha \tau_{n-1,n-1} \to\infty$, a contradiction.
	
	Now, assuming that $z_n\neq x_0$, we see that
	the nonlinear constraint
	\begin{equation*}
		\alpha\tau_{n-1,n-1} \left(f^+_{n-1} - f^+_{n-1/2}\right)+ \frac{L(\alpha-\alpha^2)}{2}\norm{x_0-z_{n}}^2 \geq 0
	\end{equation*}
	is strictly satisfied for all $\alpha\in(0,1)$. Thus, \eqref{eq:optimize_lambdas} is essentially strictly feasible in the sense of \cite[Definition 1.4.2]{ben2001lectures}. This guarantees strong duality holds for \eqref{eq:optimize_lambdas}.
\end{proof}

Next, we show that if \eqref{eq:optimize_lambdas} is unbounded, SPGM successfully finds a \emph{minimizer} of $f$.
\begin{lemma}
	\label{lem:unbounded}
	If \eqref{eq:optimize_lambdas} is unbounded, then
	\begin{equation*}
		x_{n-1/2}^+ \in\argmin_x f(x).
	\end{equation*}
\end{lemma}
\begin{proof}
	The descent lemma implies that
	\begin{equation*}
		f(x_{n-1/2}^+) \leq f_{n-1/2}^+,
	\end{equation*}
	and so it suffices to show that $f_{n-1/2}^+ \le f_\star$.
	
	\cref{lem:half_update_pc} implies that for any feasible $(\mu, \lambda)$, if $\tau = \ip{\btau,\mu} + \ip{\bo,\lambda_\star}$, then
	\[
	f_{n-1/2}^+ - f_\star \le \frac{L}{2\tau}\norm{x_0 - x_\star}^2.
	\]
	Since \eqref{eq:optimize_lambdas} is unbounded, we have that there exist feasible solutions with $\tau$ arbitrarily large. We conclude that $f_{n-1/2}^+ \le f_\star$.
\end{proof}
\subsection{Bounds on the Convergence Rate of SPGM}
\label{subsec:intermediate}
Now we are ready to present our method's formal convergence guarantees.
\begin{theorem}
	\label{thm:spgm_performance}
	Letting $x_N$ be the output of SPGM, we have that 
	\[
	f(x_N) - f(x_\star) \le \frac{L}{2\tau_{N,N}} \|x_0 - x_\star\|^2.
	\]
\end{theorem}
\begin{proof}
	We proceed by induction on $n$.
	As for the OGM, $z_1$ is an auxiliary vector for $x_0$ with pre-rate $\tau_{0,0}=2$ (see \cref{lem:OGM_initialization}). Consider any $n\in [1,N]$. By Lemma~\ref{lem:unbounded}, if the subproblem is unbounded, then SPGM reports an exact minimizer. Hence, going forward, we can assume each subproblem is bounded. Then from Lemma~\ref{lem:strict_feasibility_ogm}, we have existence of optimal $(\lambda_\star,\mu)$ since $\tau = \ip{\btau,\mu} + \ip{\bo,\lambda_\star}$ and $\btau_i >0$ for each $i$. Further, the resulting $z_{n+1/2}$ is an auxiliary vector for $x_{n-1/2}$ with pre-rate $\tau_{n-1/2}$ by \cref{lem:half_update_pc}.
	It then follows from \cref{lem:OGM_induction} that $z_{n+1}$ is an auxiliary vector for $x_n$ with pre-rate $\tau_{n,n}$, if $n<N$, and with rate $\tau_{N,N}$ if $n=N$.
	
	We conclude that
	\begin{equation*}
		f_N - f_\star \leq \frac{L}{2\tau_{N,N}}\norm{x_0-x_\star}^2.\qedhere
	\end{equation*}
\end{proof}
We will note that the value of $\tau_{N,N}$ cannot be computed without knowing all $N$ oracle responses.
The sequence of pre-rates and rates $\tau_{0,0},\dots,\tau_{N,N}$ in SPGM are computed dynamically. Specifically, $\tau_{i,i}$ depends on the first $i$ oracle responses. In order to formalize the fact that \cref{alg:SPGM} is subgame perfect, we will also need to give lower bounds on $\tau_{i,i}$ that can be computed after seeing $n\leq i$ oracle responses.
Define
the double-indexed quantity $\tau_{n,i}$ for $n\leq i$:
\begin{equation}
	\label{eq:tau_ni}
	\tau_{n,i} = \begin{cases}
		\tau_{n,n} &\text{if }i = n\\
		\tau_{n,i-1}+ \delta_i(\tau_{n,i-1})&\text{if }i >n.
	\end{cases}
\end{equation}
As an example, since $\tau_{0,0} = 2$, we have that the $\tau_{0,i}$ defined in this way coincides with the OGM sequence, $\tau_{0,0},\dots,\tau_{0,N}$. The following lemma states a monotonicity property on the $\tau_{i,j}$ and uses it to prove the first half of \cref{thm:subgamePerfect}.
\begin{lemma}
	\label{lem:n_step_guarantee}
	For any $n\in[0,N]$, it holds
	\begin{equation*}
		\tau_{0,n} \leq \tau_{1,n}\leq \dots \leq \tau_{n-1,n}\leq \tau_{n,n}.
	\end{equation*}
	In particular, if $x_N$ is the output of SPGM and $n \in [0, N]$, we have that 
	\[
	f(x_N) - f_\star \le \frac{L}{2\tau_{n,N}} \|x_0 - x_\star\|^2 \le  \frac{L}{2\tau_{0,N}} \|x_0 - x_\star\|^2.
	\]
\end{lemma}
\begin{proof}
	First, note that for all $n\in[1,N]$, the function $\delta_n(\tau)$ is monotonic increasing in $\tau$. We have that $\tau_{n-1/2}\geq \tau_{n-1,n-1}$ by \cref{lem:strict_feasibility_ogm}. Thus,
	\begin{equation*}
		\tau_{n,n} = \tau_{n-1/2} + \delta_n(\tau_{n-1/2})  \geq \tau_{n-1,n-1} + \delta_n(\tau_{n-1,n-1}) = \tau_{n-1,n}.
	\end{equation*}
	We can also chain this inequality to get,
	\begin{equation*}
		\tau_{n,n} \geq \tau_{n-1,n-1} + \delta_n(\tau_{n-1,n-1})\geq \tau_{n-2,n-1} + \delta_n(\tau_{n-2,n-1}) = \tau_{n-2,n}.
	\end{equation*}
	Repeating this argument proves the first claim.
	
	The second claim follows immediately from \cref{thm:spgm_performance} and $\tau_{N,N} \ge \tau_{n,N}$.
\end{proof}

\section{Dynamic Lower Bounds} \label{sec:lower_bounds} This section proves the optimality of SPGM in the sense of \cref{thm:subgamePerfect} assuming that $d\geq N+2$. We will assume we are given the history $\cH = \set{(x_i,f_i,g_i)}_{i=0}^{n-1}$ and the quantities $\set{(\tau_{i,i}, z_{i+1})}_{i=0}^{n-1}$ produced by SPGM. We deduce that the function $f$ belongs to the class
\begin{equation*}
	\cF^\mathcal{H} \coloneqq \set{f\colon\R^d\to\R:\, \begin{array}{l}
			f \text{ is $L$-smooth and convex}\\
			f \text{ has a minimizer $x_\star$}\\
			f(x_i)= f_i,\, \grad f(x_i) = g_i \quad\forall  i \in[0,n-1]
	\end{array}}.
\end{equation*}

From \cref{lem:n_step_guarantee}, we know that SPGM produces $x_n,\dots,x_N$ guaranteeing
\begin{equation}
	\label{eq:spgm_N_guarantee}
	f(x_N) - f(x_\star) \leq \frac{L}{2\tau_{n,N}}\norm{x_0 - x_\star}^2
\end{equation}
on any instance $f\in\cF^\mathcal{H}$.
We will prove the optimality of this bound by constructing a hard function $f_\textup{hard}\in\cF^\mathcal{H}$ so that
\begin{equation}
	\label{eq:f_hard_property}
	f_\textup{hard}(x_N) - f_\textup{hard}(x_\star) \geq \frac{L}{2\tau_{n,N}}\norm{x_0-x_\star}^2
\end{equation}
for any $x_n,\dots,x_N$ satisfying the gradient-span condition \eqref{eq:gradient_span_condition}.

\cref{subsec:overview_lower_bound} provides guiding intuition for our construction and proof strategy.
\cref{subsec:lower_bound_construction} constructs the hard function $f_\textup{hard}$.
\cref{subsec:lower_bound_proofs} then proves that (i) $f_\textup{hard}\in\cF^\mathcal{H}$ (i.e., $f_\textup{hard}$ is a valid strategy for Bob in the remaining subgame) and (ii) that no sequence of $N-n+1$ additional steps satisfying~\eqref{eq:gradient_span_condition} can produce an objective gap less than $\frac{L\|x_0-x_\star\|^2}{2\tau_{n,N}}$. Both of these properties correspond to checking certain inequalities established by~\cite{taylor2017interpolation,Drori2021OnTO}.    

\subsection{Overview of our construction}
\label{subsec:overview_lower_bound}
The construction of $f_\textup{hard}$ is motivated by the following intuition:
Let us suppose that a function $f_\textup{hard}$ satisfying \eqref{eq:f_hard_property} exists.
Clearly, this function will cause the inequality \eqref{eq:spgm_N_guarantee} to hold at equality.
On the other hand, the proof of \eqref{eq:spgm_N_guarantee} boils down to showing that \eqref{eq:spgm_N_guarantee} (after rearranging) can be written as a conic combination of some nonnegative expressions. Thus, in order for it to hold at equality, $f_\textup{hard}$ must set each of these nonnegative expressions identically to zero.
This places numerous constraints on the values of $f_\textup{hard}(x_i)$ and $\grad f_\textup{hard}(x_i)$.

This brings us to the construction: we will define $f_\textup{hard}$ to be a convex function interpolating some set of triples $\set{(x_i,f_i,g_i)}_{i\in\cI}$ where $\cI=[0,N]\cup\set{\star}$. By ensuring that $ \cH=\set{(x_i,f_i,g_i)}_{i\in[0,n]}$ is contained in this set of triples, we will ensure that $f_\textup{hard}\in\cF^\mathcal{H}$. We will then choose $\set{(x_i,f_i,g_i)}_{i\in[n+1, N]\cup\{\star\}}$ in such a way that $f_\textup{hard}$ satisfies the properties that $Q_{n-1/2,n},Q_{n,n+1},\dots,Q_{N-1,N}$, $Q_{\star,n},\dots,Q_{\star,N}$, and $\|z_{N+1}-x_\star\|^2$ are all zero.
Here, we have defined $Q_{n-1/2, j} = Q_{m, j}$, where $m = \argmin_{i \in [n-1]} f_i^+$.
It will turn out that $f_\textup{hard}$ is a hard function not only for SPGM but for any first-order method on the subgame.

\subsection{Construction}
\label{subsec:lower_bound_construction}

We will define the following hard function:
\begin{equation}\label{eq:worst-case-function}
	f_\textup{hard}(y) = \max_{\alpha\in\Delta_\mathcal{I}} \left\{\sum_{i\in \mathcal{I}} \alpha_i\left(f^+_i + \langle g_i, y-x_i^+\rangle \right) - \frac{1}{2L}\norm{\sum_{i\in \mathcal{I}} \alpha_i g_i}^2  \right\},
\end{equation}
where $\Delta_\mathcal{I}$ is the simplex in $\R^{|\cI|}$ and $(x_i,f_i,g_i)$ for $i\in[n,N]\cup\set{\star}$ will be specified in this subsection. $f_\textup{hard}$ is a modification of Drori and Taylor's hard function defined in \cite[Eq.~(1)]{Drori2021OnTO}, and is designed so that it interpolates the first order data $\{(x_i, f_i, g_i)\}_{i=0}^N$.

Before setting the values of $(x_i,f_i,g_i)$ for $i\in[n,N]\cup\set{\star}$, we must first consider the dual to the optimization problem \eqref{eq:optimize_lambdas}. 
Intuitively, an optimal solution to the dual program to \eqref{eq:optimize_lambdas} shows that $z_{n+1}$ is an auxiliary vector for $x_{n}$ with the largest certifiable pre-rate. Hence, it makes sense that our lower bound construction will be related to the optimal solution to this dual program.

The following lemma formulates the dual of \eqref{eq:optimize_lambdas} and interprets complementary slackness between \eqref{eq:optimize_lambdas} and its dual. Its proof is standard and is deferred to \cref{sec:deferred}.
\begin{lemma}
	\label{lem:dual}
	The dual program to \eqref{eq:optimize_lambdas} is
	\begin{align}
		\label{eq:optimize_lambdas_dual}
		\inf_{\xi\in\R,z\in\R^d}\set{
			\frac{L}{2\xi}\norm{x_0 - z}^2:\, \begin{array}{l}
				\bh  - LZ^\intercal z \leq  (f_{n-1/2}^+ - \xi) \btau\\
				\bq + LG^\intercal z \leq (f_{n-1/2}^+ - \xi)\mb 1\\
				\xi>0
		\end{array}}
	\end{align}
	Here, the $\inf$ can be replaced by a $\min$ as long as \eqref{eq:optimize_lambdas} is bounded. Strong duality holds between \eqref{eq:optimize_lambdas} and \eqref{eq:optimize_lambdas_dual}. If $(\mu,\lambda_\star)$ and $(\xi^*,z^*)$ are optimal solutions to \eqref{eq:optimize_lambdas} and \eqref{eq:optimize_lambdas_dual} respectively, then 
	\begin{gather*}
		x_0 + Z\mu - G\lambda_\star = z^*.
	\end{gather*}
\end{lemma}
We may assume that \eqref{eq:optimize_lambdas} is bounded\footnote{Otherwise, SPGM outputs an optimizer of $f$.}, and that $(\xi^*, z_{n+1/2})$ is an optimal solution to \eqref{eq:optimize_lambdas_dual}.

Recall that $\tau_{n-1/2}$ is the optimal value of \eqref{eq:optimize_lambdas}, that $\tau_{n,n} = \tau_{n-1/2} + \delta_n(\tau_{n-1/2})$, and that $\tau_{n,i}$ is defined by \eqref{eq:tau_ni}. We simplify our notation by setting $\delta_n = \delta_n(\tau_{n-1/2})$ and, for $i\in [n+1,N]$ $\delta_i = \delta_i(\tau_{n,i-1})$.
We also define the following auxiliary constants:
\begin{align}
	\Delta &= \|z_{n+1/2}-x_0\|^2\\
	\eta_n &= \frac{1}{2\tau_{n-1/2}\delta_n}\\ 
	\eta_i &= \frac{1}{2\tau_{n,i-1}\delta_i}\left(1 + \sum_{j=n}^{i-1}\delta_j^2\eta_j\right)\qquad\forall i\in[n+1,N].
\end{align}

We are now ready to define $(x_i,f_i,g_i)$ for $i\in[n,N]\cup\set{\star}$. We first set $g_\star = 0$.  We then define $g_n,\dots,g_{N}$ to have the following inner products:
\begin{equation*}
	\langle g_i, g_j \rangle = \eta_i L^2\Delta 1_{i = j} \text{ for }i \in [N], j \in [n,N].
\end{equation*}
That is $g_j \in \range(G)^{\perp}$ for all $j \in [n,N]$ and are also mutually orthogonal with $\|g_i\|^2=\eta_i L^2\Delta$. 
This is possible by the assumption that $d\geq N+2$.
Note that the choice to make each subsequent $g_i$ orthogonal to all previously seen gradients is also made in \cite{drori2017exact}, and intuitively is the ``least informative'' choice of gradient for an algorithm to encounter.

We then define $x_i$ and $z_{i+1}$ for $n \le i \le N$ via a mutual recursion:
\begin{align*}
	x_n &= \frac{\tau_{n-1/2}}{\tau_{n,n}}x_{n-1/2}^++\frac{\delta_n}{\tau_{n,n}}z_{n+1/2}\\
	z_{n+1} &= z_{n+1/2} - \frac{\delta_n}{L} g_n\\
	x_i &= \frac{\tau_{n,i-1}}{\tau_{n,i}}x_{i-1}^+ +\frac{\delta_i}{\tau_{n,i}}z_i\qquad\text{for all }i=n+1,\dots,N\\
	z_{i+1}&= z_{i}- \frac{\delta_i}{L} g_i\qquad\text{for all }i=n+1,\dots,N
\end{align*}
In words, the quantities $x_i$ and $z_{i+1}$ for $N \ge i \ge n$ are obtained by simulating SPGM under the assumption that in each iteration $i\geq n$, we 
see the ``worst-case'' behavior $\tau_{n,i}$ (see \cref{lem:n_step_guarantee}).

We finally define the values of $f_i$ for $i \in [n, N] \cup \{\star\}$ as well as $x_\star$. 
\begin{align}
	f_\star &= f^+_{n-1/2} - \xi^*\\
	x_\star &= z_{N+1}\\
	f_i &= f_\star + \frac{L\Delta}{2}\left(2\delta_i - 1\right)\eta_i\qquad\text{for all }i=n,\dots,N.
\end{align}
In particular,
\begin{equation*}
	f_i^+ - f_\star = L\Delta\left(\delta_i - 1\right)\eta_i\qquad\text{for all }i=n,\dots,N.
\end{equation*}
Note that by strong duality,
\begin{equation*}
	\tau_{n-1/2} = \frac{L}{2\xi^*}\norm{x_0 - z_{n+1/2}}^2 = \frac{L\Delta}{2\xi^*}.
\end{equation*}


\subsection{Lower bound guarantees}
\label{subsec:lower_bound_proofs}

Let $f_\textup{hard}$ and $\set{(x_i,f_i,g_i)}_{i\in[0,N]\cup\set{\star}}$ be constructed according to \cref{subsec:lower_bound_construction}. The following propositions establish relevant properties of this construction.

\begin{proposition} \label{prop:interpolating}
	$f_\textup{hard}$ is $L$-smooth, convex, and for all $i\in[0,N]\cup\set{\star}$, has
	\begin{align*}
		f_\textup{hard}(x_i)= f_i,\qquad \grad f_\textup{hard}(x_i)= g_i.
	\end{align*}
	Moreover, $f_N = f_\star + \frac{L}{2\tau_{n,N}}\norm{x_0-x_\star}^2$.
\end{proposition}
\begin{proposition} \label{prop:zero-chain}
	$f_\textup{hard}$ possesses the {\it zero-chain property} of~\cite{Drori2021OnTO}, that is, for any $j\in [n, N-1]$, and any
	\[
	y\in x_0 + \spann\left(\set{g_0,\dots, g_{j-1}}\right),
	\]
	we have that
	\begin{equation}
		\nabla f_\textup{hard}(y) \in \spann\set{g_0,\dots, g_{j}}. \label{eq:zero-chain-property}
	\end{equation}
\end{proposition}
These two propositions establish that any gradient-span first-order method playing the remaining iterations $n$ through $N$ against the hard instance $f_\textup{hard}$ will have its $N$th iterate lie in the affine subspace $ x_0 + \spann\left(\set{g_0,\dots, g_{N-1}}\right)$. Since $g_N$ is orthogonal to this subspace, we deduce that $f_N$ is the minimum value of $f_\textup{hard}$ on this subspace and that
any such first-order method's final iterate's objective gap must be at least $f_N - f_\star = \frac{L}{2\tau_{n,N}}\norm{x_0-x_\star}^2$.

It remains to prove \cref{prop:interpolating}, \cref{prop:zero-chain}.

The following lemmas state useful inequalities and identities. Their proofs are deferred to Appendix~\ref{app:proof_of_eta_identities}.

\begin{lemma}
	\label{lem:eta_identities}
	The following bounds hold for all $i\in[n,N]$:
	\begin{equation*}
		0<\frac{1}{2\tau_{n,i}(\delta_i-1)} \leq \eta_i\leq \frac{1}{2\tau_{n-1/2}\delta_i}.
	\end{equation*}
	Additionally, if $n\leq i < j \leq N$, then
	\begin{gather*}
		\tau_{n,i} (\delta_i - 1)\eta_i \leq \tau_{n,j}(\delta_j-1)\eta_j\qquad \text{and}\qquad 
		\delta_j\eta_j\leq (\delta_i-1)\eta_i.
	\end{gather*}
\end{lemma}


\begin{lemma}
	\label{lem:crossterm2}
	Suppose $j<i$ with $i\in[n,N]$. Then,
	\begin{equation*}
		\ip{g_j, x_i^+ - x_j^+} = 
		\begin{cases}
			\left(\frac{\tau_{n,j}}{\tau_{n,i}}-1\right)(\delta_j - 1)L\eta_j\Delta
			&\text{if }j\geq n\\
			\frac{\tau_{n-1/2}}{\tau_{n,i}}\ip{g_j, x_{n-1/2}^+ - z_{n+1/2}} +  \ip{g_j, z_{n+1/2} - x_j^+}
			& \text{if }j < n.
		\end{cases}
	\end{equation*}
\end{lemma}

\begin{lemma}
	\label{lem:subopt_value}
	It holds that $ f_N - f_\star = \frac{L}{2\tau_{n,N}}\norm{x_0-x_\star}^2$.
\end{lemma}

\begin{proof}[Proof of \cref{prop:interpolating}]
	To establish $L$-smoothness, convexity and the conditions $f(x_i)= f_i, \grad f(x_i)= g_i$, it suffices to verify $Q_{i,j}\geq 0$ for all $i,j\in[0,N]\cup\set{\star}$ due to~\cite[Theorem 1]{Drori2021OnTO}. We will prove this inequality in 9 cases depending on the value of $i,j$. These cases are outlined in the following diagram, which indicates the range of indices covered by the corresponding case:
	\[
	\bordermatrix{
		& j \in [0,n-1] & j \in [n,N] & j =\star \cr
		i \in [0,n-1] & 1 &  2 & 3\cr
		i \in [n,N] & 4 &  5/6 & 7\cr
		i = \star & 8 &  9 & \cr
	}
	\]

	\paragraph{Case 1 ($i,j\in [0,n-1]$):}
	$Q_{i,j}$ involves only quantities from $\cH = \set{(x_i,f_i,g_i)}_{i=0}^{n-1}$. Thus, as $\cH$ is produced by some $L$-smooth convex function, we have that $Q_{i,j}\geq 0$ by Lemma~\ref{lem:interpolation}.
	
	\paragraph{Case 2 ($i\in[0,n-1]$ and $j\in[n,N]$):} We have
	\begin{align*}
		Q_{i,j} &= f_i^+ -f_j^+ - \ip{g_j, x_i^+ - x_j^+}\\
		&= f_i^+ - f_j^+ - \frac{1}{L}\norm{g_j}^2\\
		&\geq (f_{n-1/2}^+-f_\star) - (f_j^+ - f_\star) - \frac{1}{L}\norm{g_j}^2\\
		&= \frac{L\Delta}{2\tau_{n-1/2}} - L\Delta(\delta_j-1)\eta_j - L\Delta\eta_j\\
		&= \frac{L\Delta}{2\tau_{n-1/2}} - L\Delta\delta_j\eta_j\\
		&\geq 0.
	\end{align*}
	Here, the second line follows as $g_j$ is orthogonal to $x_i^+ - x_j$. The third line follows as $f_{n-1/2}^+ = \min_{i\in[0,n-1]}f_i^+$. The last inequality follows by \cref{lem:eta_identities}.
	
	\paragraph{Case 3 ($i\in[0,n-1]$ and $j=\star$):} We have
	\begin{align*}
		Q_{i,\star}&= f_i^+ - f_\star\\
		&\geq f_{n-1/2}^+ - f_\star\\
		&>0.
	\end{align*}
	Here, the last line follows by dual feasibility.
	
	\paragraph{Case 4 ($i\in[n,N]$ and $j\in [0,n-1]$):} We begin by writing
	\begin{align*}
		Q_{i,j}&= f_i^+ - f_j^+ - \ip{g_j,x_i^+ - x_j^+}\\
		&= f_{n-1/2}^+ -(f_{n-1/2}^+ - f_i^+) - f_j^+ - \frac{\tau_{n-1/2}}{\tau_{n,i}}\ip{g_j, x_{n-1/2}^+ - z_{n+1/2}} -  \ip{g_j, z_{n+1/2} - x_j^+}\\
		&= \frac{\tau_{n,i}-\tau_{n-1/2}}{\tau_{n,i}}\left(f_{n-1/2}^+- f_j^+ - \ip{g_j, z_{n+1/2} - x_j^+}\right) -(f_{n-1/2}^+ - f_i^+)\\
		&\qquad + \frac{\tau_{n-1/2}}{\tau_{n,i}}\left(f_{n-1/2}^+ - f_j^+ -\ip{g_j, x_{n-1/2}^+ -x_j^+}\right)\\
		&= \frac{\tau_{n,i}-\tau_{n-1/2}}{\tau_{n,i}}\left[f_\star- f_j^+ - \ip{g_j, z_{n+1/2} - x_j^+}\right] +\frac{\tau_{n,i}-\tau_{n-1/2}}{\tau_{n,i}}\left(f_{n-1/2}^+-f_\star\right)\\
		&\qquad  -(f_{n-1/2}^+ - f_i^+)+ \frac{\tau_{n-1/2}}{\tau_{n,i}}Q_{n-1/2,j}.
	\end{align*}
	Here, the second line uses the fact that $g_i$ is orthogonal to $g_j$ and the definition of $f_i$ and \cref{lem:crossterm2}.
	
	Now, note that the square-bracketed term in the final 
	display is nonnegative:
	\begin{equation*}
		f_\star- f_j^+ - \ip{g_j, z_{n+1/2} - x_j^+} = f^+_{n-1/2} - \xi^* - \bq_j - \left(LG^\intercal z_{n+1/2}\right)_j \geq 0
	\end{equation*}
	by the second constraint in \eqref{eq:optimize_lambdas_dual}. We also have that $Q_{n-1/2,j}\geq 0$. 
	Thus, we may continue:
	\begin{align*}
		Q_{i,j} &\geq \frac{\tau_{n,i}-\tau_{n-1/2}}{\tau_{n,i}}\left(f_{n-1/2}^+-f_\star\right)  -(f_{n-1/2}^+ - f_i^+)\\
		&= \frac{\tau_{n,i}-\tau_{n-1/2}}{\tau_{n,i}}\frac{L\Delta}{2\tau_{n-1/2}} -\frac{L\Delta}{2}\left(\frac{1}{\tau_{n-1/2}} - 2\left(\delta_i - 1\right)\eta_i\right)\\
		&= \frac{L\Delta}{2}\left(2\left(\delta_i - 1\right)\eta_i - \frac{1}{\tau_{n,i}}\right)\\
		&\geq 0.
	\end{align*}
	The last line follows from \cref{lem:eta_identities}.
	
	\paragraph{Case 5 ($i,j\in [n,N]$ with $i<j$):} 
	We have
	\begin{align*}
		Q_{i,j} &= f_i^+ - f_j^+ - \ip{g_j, x_i^+ - x_j^+}\\
		&= f_i^+ - f_j^+ - \frac{1}{L}\norm{g_j}^2\\
		&= L\Delta\left((\delta_i-1)\eta_i - \delta_j\eta_j\right)\\
		&\geq 0.
	\end{align*}
	Here, the last line follows from \cref{lem:eta_identities}.
	
	\paragraph{Case 6 ($i,j\in[n,N]$ with $i > j$):}
	We have
	\begin{align*}
		Q_{i,j}&= f_i^+ - f_j^+ - \ip{g_j, x_i^+ - x_j^+}\\
		&= L\Delta((\delta_i-1)\eta_i - (\delta_j - 1)\eta_j) - \left(\frac{\tau_{n,j}}{\tau_{n,i}}-1\right)(\delta_j - 1)L\eta_j\Delta\\
		&= \frac{L\Delta}{\tau_{n,i}}\left(\tau_{n,i}(\delta_i-1)\eta_i - \tau_{n,j}(\delta_j-1)\eta_j\right)\\
		&\geq 0.
	\end{align*}
	Here, the second line follows from \cref{lem:crossterm2} and the last line follows from \cref{lem:eta_identities}.
	
	\paragraph{Case 7 ($i\in[n,N]$ and $j= \star$):} We have
	\begin{align*}
		Q_{i,\star}&= f_i^+ - f_\star\\
		&= L\Delta(\delta_i-1)\eta_i\\
		&\geq 0.
	\end{align*}
	
	\paragraph{Case 8 ($i=\star$ and $j\in[0,n-1]$):} 
	We have
	\begin{align*}
		Q_{\star,j}&= f_\star - f_j^+ - \ip{g_j,x_\star - x_j^+}\\
		&= f_\star - f_j^+ - \ip{g_j, z_{n+1/2} - x_j^+}\\
		&= f_{n-1/2}^+ - \xi^* - \bq_j - (LG z_{n+1/2})_j\\
		&\geq 0.
	\end{align*}
	Here, the second line follows as $x_\star = z_{n+1/2} -\sum_{\ell = n}^N \frac{\delta_\ell}{L} g_\ell$ where $g_\ell$ is orthogonal to $g_j$ for all $\ell\in[n,N]$.
	The inequality follows from the second constraint of \eqref{eq:optimize_lambdas_dual}.
	
	\paragraph{Case 9 ($i = \star$ and $j \in [n,N]$):} 
	We have
	\begin{align*}
		Q_{\star,j} &= f_\star - f_j^+ - \ip{g_j, x_\star - x_j^+}\\
		&=  f_\star - f_j^+ + \frac{\delta_j}{L} \norm{g_j}^2 - \frac{1}{L}\norm{g_j}^2\\
		&=  -L\Delta(\delta_j - 1)\eta_j + 
		L\Delta \delta_j\eta_j - L\Delta\eta_j\\
		&= 0.\qedhere
	\end{align*}
	
	%
\end{proof}

\begin{proof}[Proof of \cref{prop:zero-chain}]
	Conditions ensuring that \eqref{eq:worst-case-function} satisfies \eqref{eq:zero-chain-property} for $j\in [n,N-1]$ are given by Drori and Taylor~\cite[Theorem 3]{Drori2021OnTO}. In particular, it suffices to show
	\begin{align*}
		&\langle g_i, g_j \rangle = \langle g_i, g_\ell \rangle, \qquad \forall j \in[n, N-1], i\in [0,j-1], \ell\in [j+1,N]\\
		&Q_{i,\ell} - Q_{i,j} + \langle g_\ell - g_j, x_i -x_0\rangle \geq 0,\qquad \forall i\in [0,N], \forall n\leq j < \ell\leq N\\
		&g_j \text{ is linearly separable from } \{g_{j+1},\dots,g_N\}, \qquad \forall j\in [n,N-1].
	\end{align*}
	The first and third conditions above hold immediately as the considered gradients are all nonzero and orthogonal to each other. Hence, we only need to verify the second condition.
	Let $n\leq j<\ell\leq N$ and let $i\in[0,N]$.
	We break this proof into different cases depending on the ordering of $i,\,j,\,\ell$.
	
	
	First, suppose $i< j<\ell$. Since $\langle g_\ell - g_j, x_i -x_0\rangle = 0$, we have
	\begin{align*}
		Q_{i,\ell} - Q_{i,j} + \langle g_\ell - g_j, x_i -x_0\rangle &= 
		f_j^+ - f_\ell^+ - \ip{g_\ell, x_i^+ - x_\ell^+}
		+\ip{g_j, x_i^+ - x_j^+}\\
		&=f_j^+ - f_\ell^+ - \frac{1}{L}\norm{g_\ell}^2 + \frac{1}{L}\norm{g_j}^2\\
		&= L\Delta(\delta_j\eta_j - \delta_\ell\eta_\ell) \geq 0,
	\end{align*}
	which is nonnegative as $\delta_j\eta_j \geq \delta_\ell\eta_\ell$ by Lemma~\ref{lem:eta_identities}.
	
	Next, suppose $i=j<\ell$. Again, $\langle g_\ell - g_j, x_i -x_0\rangle = 0$, so
	\begin{equation*}
		Q_{i,\ell} - Q_{i,j} + \langle g_\ell - g_j, x_i -x_0\rangle = Q_{i,\ell} -0+0\geq 0.
	\end{equation*}
	
	Suppose $j < i < \ell$. We may rewrite $\langle g_\ell - g_j, x_i -x_0\rangle = -\langle g_j, x_i^+ -x_j^+\rangle + \frac{1}{L}\norm{g_j}^2$. Thus,
	\begin{align*}
		&Q_{i,\ell} - Q_{i,j} + \langle g_\ell - g_j, x_i -x_0\rangle\\
		&= f_j^+ - f_\ell^+ - \ip{g_\ell, x_i^+ - x_\ell^+} + \frac{1}{L}\norm{g_j}^2\\
		&= L\Delta(\delta_j\eta_j - \delta_\ell\eta_\ell)\\
		&\geq 0
	\end{align*}
	where the final inequality uses Lemma~\ref{lem:eta_identities}.
	
	Next, suppose $j < i = \ell$. We may rewrite $\langle g_\ell - g_j, x_i -x_0\rangle = -\langle g_j, x_i^+ -x_j^+\rangle + \frac{1}{L}\norm{g_j}^2$. Thus, again by Lemma~\ref{lem:eta_identities},
	\begin{align*}
		&Q_{i,\ell} - Q_{i,j} + \langle g_\ell - g_j, x_i -x_0\rangle\\
		&= -f_i^+ + f_j^+ + \frac{1}{L}\norm{g_j}^2\\
		&= -L\Delta(\delta_i-1)\eta_i + L\Delta\delta_j\eta_j \geq 0.
	\end{align*}
	Finally, suppose $j < \ell < i$. Since $\langle g_\ell - g_j, x_i -x_0\rangle = \langle g_\ell, x_i^+ -x_\ell^+\rangle - \frac{1}{L}\norm{g_\ell}^2 - \langle g_j, x_i^+ -x_j^+\rangle + \frac{1}{L}\norm{g_j}^2$. Thus, by Lemma~\ref{lem:eta_identities},
	\begin{align*}
		&Q_{i,\ell} - Q_{i,j} + \langle g_\ell - g_j, x_i -x_0\rangle\\
		&= f_j^+ - f_\ell^+ - \frac{1}{L}\norm{g_\ell}^2 + \frac{1}{L}\norm{g_j}^2\\
		&= L\Delta(\delta_j-1)\eta_j - L\Delta(\delta_\ell-1)\eta_\ell - L\Delta\eta_\ell + L\Delta \eta_j\\
		&= L\Delta(\delta_j\eta_j-\delta_\ell\eta_\ell)\geq 0.\qedhere
	\end{align*}
\end{proof}

\section{Limited-Memory SPGM}
\label{sec:lSPGM}
\subsection{A limited memory algorithm}
Running SPGM as defined in~\cref{alg:SPGM} requires storing a history of all previous iterates and auxiliary vectors explicitly in order to produce an update. In practice, this overhead may be overly costly, and it would be preferable to store only a fixed amount of history. Our approach can be directly modified to only store a fixed number of vectors, while still being practically performant and maintaining theoretical guarantees on its performance (though this method will no longer be subgame perfect).

To see this, note that
the main result ensuring the correctness of the convergence rate of SPGM, \cref{lem:half_update_pc}, only requires a feasible solution to \eqref{eq:optimize_lambdas}. In particular, if we modify \eqref{eq:optimize_lambdas} to have the additional constraints that $\mu_i = \lambda_{\star, i} = 0$ for all $i \in [0, \dots, n - k-1]$, and use the output of the resulting subproblem to make our updates, \cref{lem:half_update_pc} will continue to hold. The effect of these constraints is to remove the dependence of the algorithm on the information gained in iterations prior to iteration $n-k$.

Formally, we will define the limited memory analogue of \eqref{eq:optimize_lambdas} given a fixed budget $k \in [0, \dots, N]$ of memory.
We will use $\hat{\cdot}$ to denote the restriction of a vector to its last $k$ rows or a matrix to its last $k$ columns. We define
\begin{equation}
	\label{eq:l_optimize_lambdas}
	\sup_{\mu,\lambda_\star\in\R^{k}}\set{\ip{\hat{\btau},\mu} + \ip{\bo,\lambda_\star} :\, \begin{array}{l}
			\frac{L}{2}\left(\norm{z}^2 - \norm{x_0}^2\right)\leq 
			\ip{\mu,\hat{\bh}-f^+_{n-1/2}\hat{\btau}} + \ip{\lambda_\star, \hat{\bq} - f^+_{n-1/2}\mb 1}\\
			\tilde{z} = \hat{Z}\mu - \hat{G}\lambda_\star + x_0\\
			\mu,\lambda_\star\geq 0
	\end{array}}.
	\tag{k-SPGM-Subproblem}
\end{equation}
Here, $f_{n-1/2}^+ = \min_{i\in[n-k,n-1]}f_i^+$ and $x_{n-1/2}$ is defined to be $x_{n-1/2} = x_m$ for some $m\in\argmin_{i\in[n-k,n-1]}f_i^+$.
In \cref{alg:k_SPGM}, we define $k$-SPGM identically to SPGM, except with \eqref{eq:optimize_lambdas} replaced by \eqref{eq:l_optimize_lambdas}:
\begin{algorithm}[H]
	\caption{$k$-SPGM}
	\label{alg:k_SPGM}
	Given $L$-smooth convex function $f$, initial iterate $x_0$, iteration budget $N$, memory bound $k > 1$:
	\begin{itemize}
		\item Define $\tau_{0,0} = 2$ and $z_1 = x_0 - \frac{2}{L}g_0$
		\item For $n = 1,\dots,N$
		\begin{enumerate}
			\item Reveal $g_{n-1} = \grad f(x_{n-1})$ and update first-order information set $\cH\coloneqq \set{(x_i,f_i,g_i)}_{i={n-k}}^{n-1}$
			\item 
			Define $x_{n-1/2} = x_m$ where $m \in \argmin_{i\in[n-k,n-1]}f^+_i$ 
			and let $\tau_{n-1/2}$ be the optimal value of \eqref{eq:l_optimize_lambdas}.
			\item If \eqref{eq:l_optimize_lambdas} is unbounded, output $x_{n-1/2}^+$. Otherwise, define $z_{n+1/2}$ by \eqref{eq:x_half_def}.
			\item Define
			\begin{align*}
				\tau_{n,n} & \coloneqq \tau_{n-1/2} +\delta_n(\tau_{n-1/2}),\\
				x_n & \coloneqq \frac{\tau_{n-1/2}}{\tau_{n,n}}x_{n-1/2}^+ + \left(1 - \frac{\tau_{n-1/2}}{\tau_{n,n}}\right) z_{n+1/2},\\
				z_{n+1} & \coloneqq z_{n+1/2} - \frac{\delta_n(\tau_{n-1/2})}{L} g_n,
			\end{align*}
			where $\delta_n(\cdot)$ is defined in \eqref{eq:delta}.
		\end{enumerate}
	\end{itemize}
\end{algorithm}

Because any feasible solution to \eqref{eq:l_optimize_lambdas} is also feasible for \eqref{eq:optimize_lambdas} with the same objective value, \cref{lem:half_update_pc} can be used to prove dynamic convergence rates for $k$-SPGM:
\begin{theorem}
	Let $x_0, \dots, x_N \in \R^d$ and $\tau_{0,0}, \dots, \tau_{N,N} \in \R$ be as defined in \Cref{alg:k_SPGM}.
	For any $1 \le n \le N$, we have that 
	\[
	f(x_N) - f(x_\star) \le \frac{L}{2\tau_{n,N}} \|x_0 - x_\star\|^2 \leq \frac{L}{2\tau_{0,N}}\norm{x_0 -x_\star}^2,
	\]
	where $\tau_{n,N}$ is defined in terms of $\tau_{n,n}$ as in \eqref{eq:tau_ni}.
\end{theorem}

\subsection{Running costs}
Running $k$-SPGM requires storing the vectors $\hat\btau,\hat\bh,\hat\bq\in\R^k$, the matrices $\hat{Z},\hat{G}\in\R^{d\times k}$, and the points $x_{n-k},\dots,x_{n-1}\in\R^d$. 
Note that by substituting out the definition of $\tilde{z}$ in \eqref{eq:l_optimize_lambdas}, one can rewrite
\begin{align*}
	\norm{\tilde{z}}^2 - \norm{x_0}^2 &= \norm{\hat Z \mu - \hat G \lambda_\star + x_0}^2 - \norm{x_0}^2\\
	&= \mu^\intercal(\hat Z^\intercal\hat Z)\mu +
	\lambda_\star^\intercal(\hat G^\intercal\hat G)\lambda_\star + 
	2 \mu^\intercal(\hat Z^\intercal\hat G)\lambda_\star
	+ 2(x_0^\intercal \hat Z)\mu - 2(x_0^\intercal \hat G) \lambda_\star
\end{align*}
to express this quantity in terms of $k\times k$-dimensional matrices and $k$-dimensional vectors only.
Thus, we will dynamically maintain the inner product matrices $\hat{Z}^\intercal \hat{Z}$, $\hat{G}^\intercal \hat{G}$, $\hat{G}^\intercal \hat{Z}\in\R^{k\by k}$ and the vectors $\hat Z^\intercal x_0$ and $\hat G^\intercal x_0$. This amounts to a total storage overhead of order $O(dk)$ assuming that $k\ll d$.
The computational overhead in each iteration involves updating these matrices and vectors in $O(dk)$ time and solving the convex quadratic optimization problem \eqref{eq:optimize_lambdas} in $2k$ variables, addressable by an interior point method via $O(\sqrt{k})$ $k\times k$ linear system solves. Assuming that $k\ll d$, the computational cost of solving the low-dimensional optimization problem is dominated by $O(dk)$.

\section{Implementations and Numerical Results}
\label{sec:numerical}
\subsection{Numerical results}
This section evaluates the computational performance of SPGM on unconstrained smooth minimization problems involving random and real data.
We implement SPGM in Julia and use MOSEK~\cite{mosek} via JuMP~\cite{Lubin2023} to solve \eqref{eq:optimize_lambdas} in each iteration of SPGM.
All experiments are run on an Intel i7 processor with 64GB of memory. A GitHub repository containing our implementations and preliminary numerical experiments is available at
\begin{center}
	\url{https://github.com/ootks/SubgamePerfectGradientMethod}
\end{center}

We compare the performance of SPGM and its limited memory variant with the following benchmarks: gradient descent (GD) with constant stepsize, i.e., $x_{n} = x_{n-1} - \frac{1}{L}\nabla f(x_{n-1})$,
the Optimized Gradient Method (OGM)~\cite{Kim2016optimal}, Universal Fast Gradient Method~\cite{nesterov2015universal}, and BFGS and its limited-memory variant (L-BFGS)~\cite{liu1989limited}.
We implement GD, OGM, and UFGM directly in Julia and use the implementations of BFGS and L-BFGS (restricted to a memory of ten past iterations) provided in \texttt{Optim.jl}~\cite{mogensen2018optim}. 

GD, OGM, and UFGM 
form one collection of baseline first-order methods against which we compare SPGM. These are algorithms with well-established non-asymptotic convergence theory. Overall, we find SPGM with either full memory or a memory limited to size $k=10$ (denoted SPGM-10) consistently outperforms GD, OGM, and UFGM by exploiting the non-adversarialness of the given problem instances. The performance of UFGM is strongest among these three comparison algorithms. We attribute this to its adaptive selection of a smoothness constant $L_n$ at each iteration as a function of observed first-order information. When the given theoretically supported global smoothness constant $L$ is much larger than the local smoothness in a neighborhood of the iterates or the minimizer, GD, OGM, and SPGM all take more conservative steps. UFGM adapts this constant to speed up. This mechanism of learning from observed first-order data to improve an estimate of $L$ is orthogonal to this work's core thrust, learning from observed first-order data to (exactly optimally) improve an algorithm's performance guarantee. The development of future, practical methods incorporating these two distinct forms of learning from data is left as an important future direction.

BFGS and L-BFGS form a second collection of baseline first-order methods against which we compare SPGM.
Although BFGS and L-BFGS lack strong non-asymptotic theory, they are known to provide state-of-the-art practical performance. These methods (and quasi-Newton methods more generally) are substantially similar to our subgame perfect methods in their computational requirements: They are gradient span algorithms and maintain a history of gradients and iterates used at each iteration to compute each update.
SPGM and its associated guarantee $1/\tau_{n,n}$ stay fairly competitive with the convergence of BFGS and L-BFGS. Figure~\ref{fig:representative-experiment} shows the scaled objective gap $(f(x_n)-f(x_\star))/\frac{L}{2}\|x_0-x_\star\|^2$ and $1/\tau_{n,n}$ for each method on two representative problem instances from the larger experiments below with random and real problem data respectively.
\begin{figure}[h]
	\centering{\includegraphics[width=0.4\textwidth]{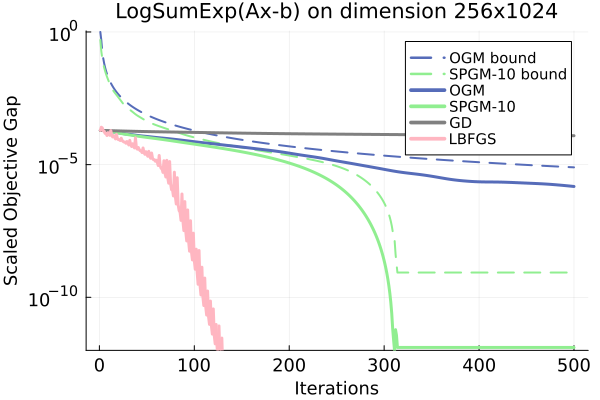}\includegraphics[width=0.4\textwidth]{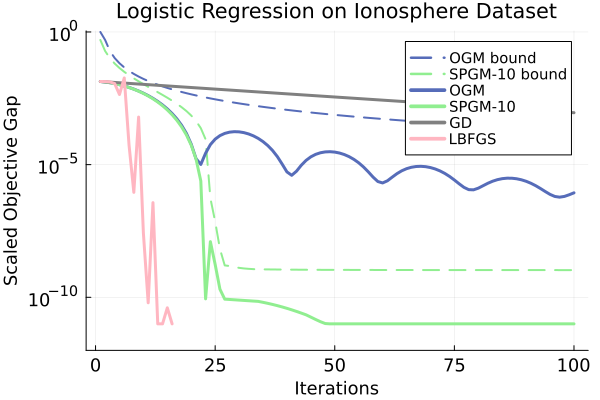}}
	\caption{Two representative plots of convergence of $(f(x_n)-f(x_\star))/\frac{L}{2}\|x_0-x_\star\|^2$ for the considered methods.  Left: A randomly generated instance of minimizing the logSumExp function~\eqref{eq:logSumExp} with dimensions $d=256,m=1024$. Right: An instance of the logistic regression~\eqref{eq:logisticRegression} using the LIBSVM dataset ``\texttt{ionosphere}'' with $d=34,m=351$. On these instances, the limited memory versions of SPGM and BFGS had nearly identical convergence to their full memory counterparts and are omitted.}\label{fig:representative-experiment}	
\end{figure}

The optimal subgame payoff, $\tau_{n,N}$, of the remaining game after $n$ iterations
is a function of $n$ that increases with each suboptimal or non-adversarial response from the first-order oracle. Figure~\ref{fig:payoffs} plots $1/\tau_{n,N}$ for SPGM and SPGM-10, as a function of $n$. We also plot the constant non-adaptive guarantee of OGM, $1/\tau_{0,N}$, for reference.

\begin{figure}[h]
	\centering{\includegraphics[width=0.4\textwidth]{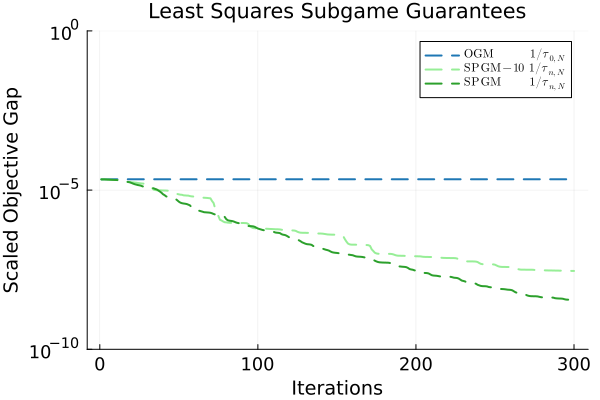}\includegraphics[width=0.4\textwidth]{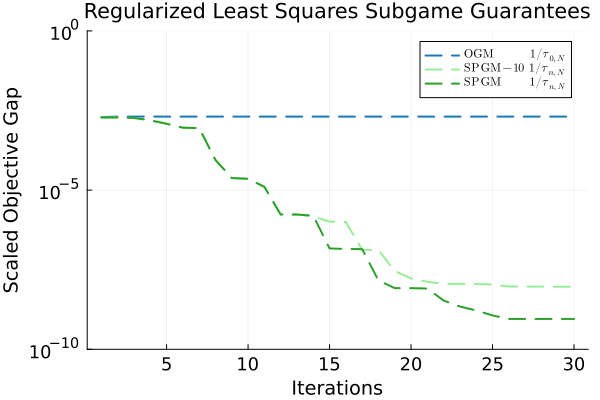}}
	\caption{Two representative plots of $1/\tau_{n,N}$ as a function of $n$ as SPGM-10 and SPGM are run on Least Squares Regression problems of dimension $d=512,m=2048$. OGM displays its constant guarantee, $1/\tau_{0,N}$, providing the baseline, non-adaptive guarantee.  Left instance is of the form~\eqref{eq:BasicLeastSquares} with $N=300$ and right instance is of the form~\eqref{eq:LSR_HuberL1} with $N=30$. }\label{fig:payoffs}	
\end{figure}

\subsection{A sample of randomly generated problem instances}
We consider six different families of smooth convex minimization problems all parameterized by $A\in\mathbb{R}^{m \times d}$,  $b\in\mathbb{R}^{m}$, and $x_0\in\mathbb{R}^d$. For each problem family, we draw seven problem instances $A,b,x_0$ sampled with i.i.d.~normal entries, each with fixed, different problem dimensions $d=8,16,32,\dots, 512$ and $m=4d$. Together these provide a collection of $42$ synthetic problem instances to compare performance across.
The problem instances considered fall within two general categories.

First, we consider (regularized) least squares regression problems defined in the following four ways:
\begin{align}
	&\min_{x\in\mathbb{R}^d} f(x)=\frac{1}{m}\|Ax-b\|^2_2, \label{eq:BasicLeastSquares}\\
	&\min_{x\in\mathbb{R}^d} f(x)= \frac{1}{m}\|Ax-b\|^2_2 + \frac{1}{2}\|x\|^2_2,\\
	&\min_{x\in\mathbb{R}^d} f(x)= \frac{1}{m}\|Ax-b\|^2_2 + h_{100}(\|x\|_2),\\
	&\min_{x\in\mathbb{R}^d} f(x)= \frac{1}{m}\|Ax-b\|^2_2 + \sum_{i=1}^d h_{100}(|x_i|)\label{eq:LSR_HuberL1}
\end{align}
where $h_L(r)$ takes the value $\frac{L}{2}r^2$ if $r\leq 1$ and the value $Lr - \frac{L}{2}$ otherwise, denotes the Huber function. Note $h_{100}(\|x\|_2)$ provides a smooth approximation of $100 \|x\|_2$ and $\sum_{i=1}^d h_{100}(|x_i|)$ provides a smooth approximation of regularization by $100\|x\|_1$. While the first two definitions have constant Hessian, the latter may not have Hessians exist everywhere.
Second, we consider smoothed approximations of minimizing a finite maximum $\max_i \{a_i^Tx - b_i\}$, defined in the following two ways:
\begin{align}
	&\min_{x\in\mathbb{R}^d} f(x)= \log\left(\sum_{i=1}^m \exp(a_i^Tx - b_i)\right), \label{eq:logSumExp}\\
	&\min_{x\in\mathbb{R}^d} f(x)= \rho_{\max}(Ax-b) \label{eq:MoreauMax}
\end{align}
where $\rho_{\max}(z) = \min_{z'}\{\max_i z'_i + \frac{1}{2}\|z'-z\|^2_2\}$ denotes the Moreau envelope of the max function. Note while the first formulation is analytic, the second is not even twice differentiable everywhere.

Sampling seven instances from each of these six classes provides our first test set of $42$ problem instances. Figure~\ref{fig:random-experiment} shows the performance of each considered method, displaying at each iteration $n$, the fraction of these $42$ problem instances for which the method has reached accuracy
$$ \frac{f(x_n)-f(x_\star)}{\frac{L}{2}\|x_0-x_\star\|^2} \leq \{10^{-3},\ 10^{-6},\ 10^{-9}\}$$
corresponding to the method reaching a low, medium, or high level of (appropriately rescaled) accuracy.
\begin{figure}
	\includegraphics[width=0.33\textwidth]{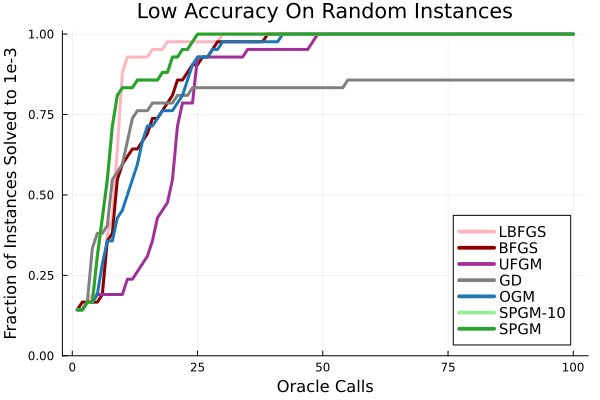}\includegraphics[width=0.33\textwidth]{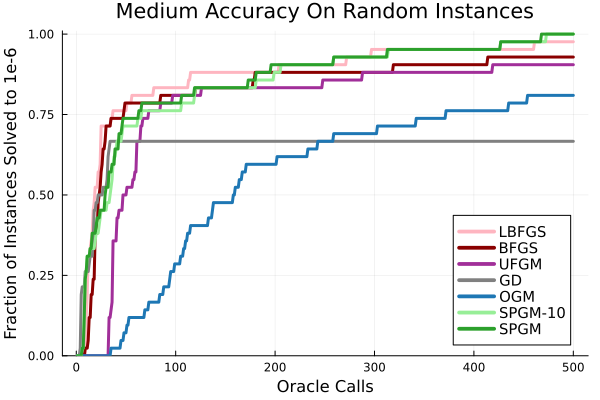}\includegraphics[width=0.33\textwidth]{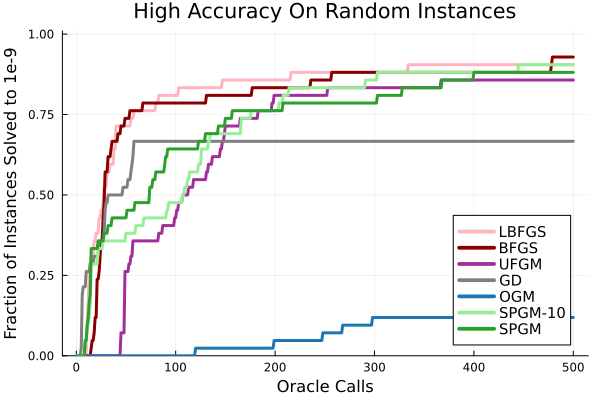}
	\caption{Performance comparison over $42$ randomly generated instances for the problems~\eqref{eq:BasicLeastSquares}--\eqref{eq:MoreauMax}. From left to right, the target accuracy ranges as $10^{-3},\ 10^{-6},\ 10^{-9}$. }\label{fig:random-experiment}
\end{figure}

The per-iteration cost of SPGM-10 as problem dimension grows should be dominated by gradient oracle costs as our subproblem only depends on the constant memory size $k$, asymptotically matching the order of our comparison methods. Our implementation solves these subproblems with independent calls to Mosek at each iteration incurring nontrivial, but constant, computational overheads. We expect with optimized numerical linear algebra and warm-starting/preconditioning subproblems this overhead could be substantially reduced. Table~\ref{tab:per-iter} shows the average seconds per-iteration in the above experiment before reaching accuracy $10^{-9}$ at each dimension $32,\dots, 512$. Per-iteration costs for GD and UFGM are similar to OGM. We see SPGM-10's cost scales linearly in $d$, with our direct implementation being roughly ten times slower than \texttt{Optim.jl}'s LBFGS implementation.

\begin{table}	\centering
	\begin{tabular}{|c|c|c|c|c|c|}
		\hline
		Method &  $d=32$ & $d=64$ & $d=128$ & $d=256$ & $d=512$\\\hline
		LBFGS & $5.00\times 10^{-5}$ & $1.87\times 10^{-4}$ & $4.20\times 10^{-4}$ & $4.42\times 10^{-4}$ & $2.75\times 10^{-3}$\\
		BFGS & $4.01\times 10^{-5}$ & $1.94\times 10^{-4}$ & $4.19\times 10^{-4}$ & $4.91\times 10^{-4}$ & $2.90\times 10^{-3}$\\
		OGM & $6.53\times 10^{-5}$ & $2.11\times 10^{-4}$ & $3.74\times 10^{-4}$ & $9.91\times 10^{-4}$ & $6.33\times 10^{-3}$\\
		SPGM\text{-}10 & $3.03\times 10^{-3}$ & $3.83\times 10^{-3}$ & $5.81\times 10^{-3}$ & $1.12\times 10^{-2}$ & $2.17\times 10^{-2}$\\
		SPGM & $9.54\times 10^{-3}$ & $1.30\times 10^{-2}$ & $2.97\times 10^{-2}$ & $7.78\times 10^{-2}$ & $2.12\times 10^{-1}$\\ \hline
		
	\end{tabular}
	\caption{Average seconds per-iteration of methods on our synthetic problem test set.}\label{tab:per-iter}
\end{table}

\subsection{Regression problem instances derived from real data}
To provide a more realistic numerical survey, we next consider performance on two types of regression problems with data from the LIBSVM dataset~\cite{libsvm}. We consider six regularized least squares regressions of the form~\eqref{eq:LSR_HuberL1} with data from the ``\texttt{housing}'', ``\texttt{mpg}'', ``\texttt{pyrim}'', ``\texttt{space\_ga}'', ``\texttt{triazines}'', ``\texttt{bodyfat}'' datasets. Further, we consider six regularized logistic regression problems defined as
\begin{align}
	&\min_{x\in\mathbb{R}^d} f(x)= \frac{1}{m}\sum_{i=1}^m\log( 1 + \exp(b_i \cdot a_i^T x)) + \frac{1}{2m}\|x\|^2_2, \label{eq:logisticRegression}
\end{align}
using normalized data $A,b$ from the classification datasets ``\texttt{australian}'', ``\texttt{diabetes}'', ``\texttt{heart}'', ``\texttt{ionosphere}'', ``\texttt{splice}'', and ``\texttt{sonar}''. The performance on these twelve instances is presented in Figure~\ref{fig:real-experiment}.
\begin{figure}
	\includegraphics[width=0.33\textwidth]{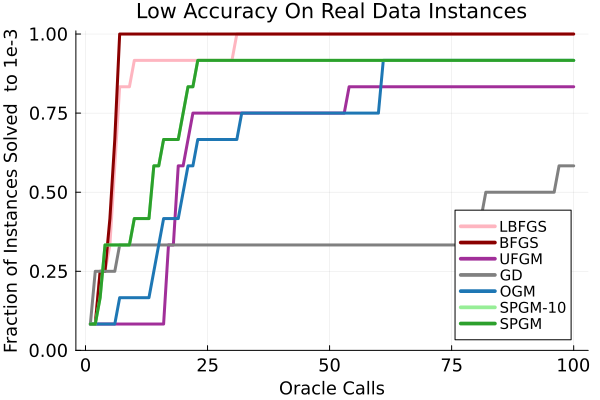}\includegraphics[width=0.33\textwidth]{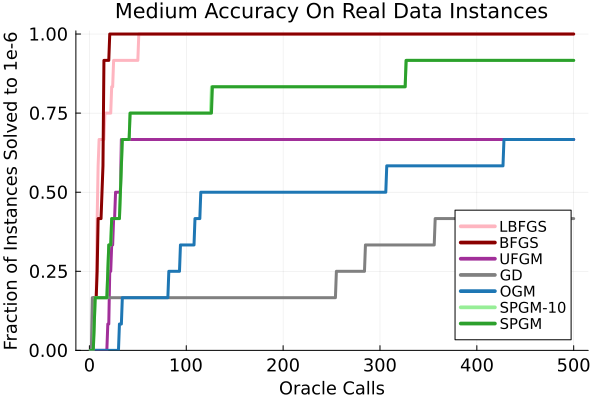}\includegraphics[width=0.33\textwidth]{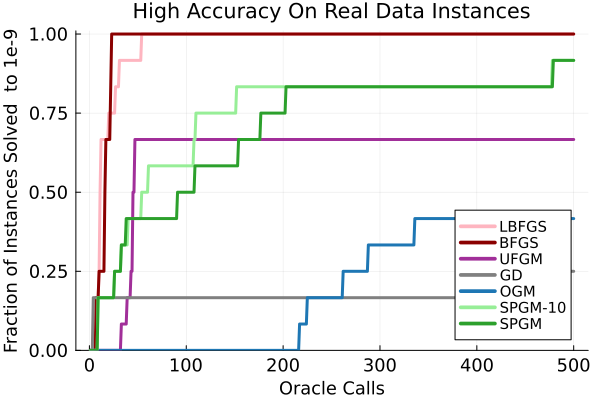}
	\caption{Performance comparison over twelve instances derived from LIBSVM data for the regression problems~\eqref{eq:LSR_HuberL1} and~\eqref{eq:logisticRegression}. From left to right, the target accuracy ranges as $10^{-3},10^{-6},10^{-9}$. Often, the performance of SPGM-10 matched that of SPGM with full memory, and hence, the lines overlap. }\label{fig:real-experiment}
\end{figure}

We see that SPGM and its limited memory variant continue to provide strong performance relative to other methods with nonasymptotic theory (GD, OGM, UFGM) but lags farther behind the quasi-Newton methods. We suspect the spectral properties of Hessians in this real data problems differ in nature from those of random matrices. The Newton-type methods appear to benefit strongly from this. This indicates there may be potential for preconditioning to improve the performance of SPGM (and GD, OGM, UFGM).

\section{Conclusion and discussion}

This paper offers a framework for analyzing first-order methods beyond the classical notion of minimax optimality.
In contrast to minimax optimality, which can be thought of as a strategy in a Nash Equilibrium for the minimization ``game'', the notion of a Subgame Perfect Equilibrium requires that a first-order method optimally capitalize on any non-adversarial first-order information.
We believe this is a fundamental avenue towards deriving formal beyond-worst-case results in the design of optimization algorithms.

As a first algorithm in this space, we design the Subgame Perfect Gradient Method (SPGM) for unconstrained smooth convex minimization.
SPGM matches the performance guarantee of the Optimized Gradient Method when faced with an adversarial function, 
but \emph{improves} upon these guarantees dynamically as first-order information is revealed. We show that SPGM's dynamic convergence guarantees are optimal among algorithms that can adapt to past information.
We conjecture that subgame perfect versions of known minimax optimal first-order methods such as ITEM for smooth strongly convex optimization~\cite{Drori2021OnTO} and OGM-G for minimizing the gradient norm~\cite{Kim2021gradient} also exist and leave these extensions as future work. The potential for dynamic guarantees on gradient norm convergence would open several further practical advantages. Since the gradient norm can be computed at runtime, unlike suboptimality, this could facilitate the design of new adaptive restarting schemes for problem settings like strongly convex optimization known to benefit from restarting.\\

\noindent {\bf Acknowledgments.}
This work was supported in part by the Air Force Office of Scientific Research under award number FA9550-23-1-0531. Benjamin Grimmer was additionally supported as a fellow of the Alfred P. Sloan Foundation.

\bibliographystyle{unsrt}
\bibliography{bib}

\appendix

	\section{Deferred proofs}
\label{sec:deferred}

\begin{proof}
	[Proof of \cref{lem:OGM_initialization}]
	We expand
	\begin{align*}
		&f_\star - f_0^+ - \frac{L}{4}\norm{z_1-x_\star}^2 + \frac{L}{4}\norm{x_0-x_\star}^2\\
		&\qquad = f_\star - f_0^+ 
		- \frac{1}{L}\norm{g_0}^2
		+\ip{x_0-x_\star, g_0}\\
		&\qquad = f_\star - f_0^+ 
		-\ip{g_0, x_\star - x_0^+}.
	\end{align*}
	The expression on the final line is $Q_{\star,0}$ which is nonnegative by \cref{lem:interpolation}.
\end{proof}

\begin{proof}
	[Proof of \cref{lem:OGM_induction}]
	
	By assumption, $z$ is an auxiliary vector for $x$ with pre-rate $\tau$ so
	\begin{equation*}
		\tau (f_\star - f(x)^+) - \frac{L}{2}\norm{z-x_\star}^2 + \frac{L}{2}\norm{x_0-x_\star}^2\geq 0.
	\end{equation*}
	Let $g= \grad f(x')$.
	By \cref{lem:interpolation}, the following quantities are nonnegative
	\begin{gather*}
		f(x)^+ - f(x')^+ - \ip{g, x^+ - (x')^+}\geq 0,\\
		f_\star - f(x')^+ - \ip{g, x_\star - (x')^+} \geq 0.
	\end{gather*}
	We sum these three nonnegative expressions with weights $(1,\tau, \delta)$ to get
	\begin{equation*}
		\tau' (f_\star - f(x')^+) + \frac{L}{2}\norm{x_0-x_\star}^2 
		- \left(\frac{L}{2}\norm{z-x_\star}^2 - \ip{g, \delta (z-x_\star) - \frac{\tau'}{L} g}\right)\geq 0.
	\end{equation*}
	The remainder of the proof simplifies the parenthetical term in the two cases.
	
	First, suppose $\delta = 1 + \sqrt{1+2\tau}$. In this case $\delta^2 = 2 (\tau + \delta) = 2\tau'$. Thus, the parenthetical term is
	\begin{equation*}
		\frac{L}{2}\norm{z-x_\star}^2 - \ip{\delta g, z-x_\star} + \frac{1}{2L} \norm{\delta g}^2 = \frac{L}{2}\norm{z' - x_\star}^2.
	\end{equation*}
	We deduce that $z'$ is an auxiliary vector for $x'$ with pre-rate $\tau'$.
	
	Now, suppose $\delta = \frac{1 + \sqrt{1+4\tau}}{2}$. In this case $\delta^2 = \tau + \delta = \tau'$. Thus, the parenthetical term is 
	\begin{equation*}
		\frac{L}{2}\norm{z-x_\star}^2 - \ip{\delta g, z-x_\star} + \frac{1}{L} \norm{\delta g}^2 = \frac{L}{2}\norm{z_{n+1} - x_\star}^2 + \frac{\tau'}{2L}\norm{ g}^2.
	\end{equation*}
	We deduce that $z'$ is an auxiliary vector for $x'$ with rate $\tau'$.
\end{proof}

\begin{proof}[Proof of \cref{lem:dual}]
	Standard manipulations for computing the dual to a conic program gives the following dual program to \eqref{eq:optimize_lambdas}:
	\begin{equation*}
		\inf_{\xi, w \in\R,z\in\R^d}\set{
			w :\, \begin{array}{l}
				\begin{pmatrix}
					- LZ^\intercal & \btau\\
					LG^\intercal &  \mb 1
				\end{pmatrix}
				\begin{pmatrix} z\\\xi\end{pmatrix}
				\le \begin{pmatrix}f_{n-1/2}^+ \btau - \bh\\ 
					f_{n-1/2}^+ \mb 1 - \bq
				\end{pmatrix}\\
				w \xi \ge \frac{L}{2}\norm{x_0 - z}^2\\
				\xi, w >0
		\end{array}}.
	\end{equation*}
	This problem is easily seen to be equivalent to \eqref{eq:optimize_lambdas_dual} after performing a partial minimization in $w$.
	
	By \cref{lem:strict_feasibility_ogm}, we know that \eqref{eq:optimize_lambdas} is essentially strictly feasible so that strong duality holds and that the $\inf$ can be replaced by a minimum as long as \eqref{eq:optimize_lambdas} is bounded.
	
	Fix optimizers $(\mu, \lambda_\star)$ and $(\xi^*, z^*)$ for \eqref{eq:optimize_lambdas} and \eqref{eq:optimize_lambdas_dual} respectively. Consider
	\[
	\begin{pmatrix} \mu\\ \lambda_\star\end{pmatrix}^{\intercal}
	\begin{pmatrix}
		- LZ^\intercal & \btau\\
		LG^\intercal & \mb 1
	\end{pmatrix}
	\begin{pmatrix} z^*\\\xi^*\end{pmatrix}.
	\]
	Let $z = Z \mu - G \lambda_\star + x_0$ and $\tau = \langle \btau, \mu\rangle + \langle \mb 1, \lambda_\star\rangle$.
	
	On the one hand, we see that
	\[
	\begin{pmatrix} \mu\\ \lambda_\star\end{pmatrix}^{\intercal}
	\begin{pmatrix}
		- LZ^\intercal & \btau\\
		LG^\intercal & \mb 1
	\end{pmatrix}
	\begin{pmatrix} z^*\\\xi^*\end{pmatrix} =
	\begin{pmatrix}L(x_0 - z)\\\tau\end{pmatrix}^\intercal
	\begin{pmatrix} z^*\\\xi^*\end{pmatrix} = L \langle z^*, x_0 - z\rangle + \xi^* \tau.
	\]
	On the other hand, by feasibility of $(\mu, \lambda_\star)$ and $(\xi^*, z^*)$ in \eqref{eq:optimize_lambdas} and \eqref{eq:optimize_lambdas_dual}, we have
	\[
	\begin{pmatrix} \mu\\ \lambda_\star\end{pmatrix}^{\intercal}
	\begin{pmatrix}
		- LZ^\intercal & \btau\\
		LG^\intercal & \mb 1
	\end{pmatrix}
	\begin{pmatrix} z^*\\\xi^*\end{pmatrix} \le
	- \ip{\mu,\bh-f^+_{n-1/2}\btau} - \ip{\lambda_\star, \bq - f^+_{n-1/2}\mb 1} \le 
	\frac{L}{2} (\|x_0\|^2 - \|z\|^2).
	\]
	Combining these inequalities, we obtain
	\[
	\tau \le \frac{L}{2\xi^*}(\|x_0\|^2 - \|z\|^2 - 2 \langle z^*, x_0 - z\rangle) =  \frac{L}{2\xi^*}(\norm{x_0 - z^*}^2 - \norm{z-z^*}^2).
	\]
	By strong duality, $\tau = \frac{L}{2\xi^*}\norm{x_0 - z^*}^2$. Thus, $z = z^*$.
\end{proof}

\label{app:proof_of_eta_identities}
\begin{proof}[Proof of \cref{lem:eta_identities}]
	We will require the following identity:
	\begin{align*}
		\delta_i^2 &= \begin{cases} 
			2\tau_{n,i}&\text{if }i \in[n,N-1]\\
			\tau_{n,N}&\text{if }i=N,
		\end{cases}
	\end{align*}
	which follows from the definition of $\delta$ in~\eqref{eq:delta} and $\tau_{n,i}$ in \eqref{eq:tau_ni}.
	
	From this, we deduce
	\begin{align}
		\label{eq:taupsi_identity}
		\tau_{n,i} (\delta_i-1)&=\begin{cases}
			(\tau_{n-1/2}+ \delta_n)(\delta_n -1)&\text{if }i=n\\
			(\tau_{n,i-1} +\delta_i)(\delta_i - 1)& \text{if }i\in[n+1,N]
		\end{cases} \nonumber\\
		&=\begin{cases}
			\tau_{n-1/2}\delta_n + \tau_{n,n} & \text{if } i = n\\
			\tau_{n,i-1}\delta_i +\tau_{n,i} & \text{if } i\in [n+1,N-1]\\
			\tau_{n,N-1}\delta_N & \text{if } i=N
		\end{cases}
	\end{align}
	and for all $i\in[n,N-1]$,
	\begin{align*}
		2\tau_{n,i}\delta_{i+1}\eta_{i+1}&= 1 + \sum_{j = n}^i\delta_j^2 \eta_j =  1 + \sum_{j = n}^{i-1} \delta_j^2 \eta_j + \delta_i^2 \eta_i\\
		&= \begin{cases}
			2\eta_n\left(\tau_{n-1/2}\delta_n + \tau_{n,n} \right) & \qquad i =n\\
			2\eta_i\left(\tau_{n,i-1}\delta_i + \tau_{n,i} \right)&\qquad i\in[n+1,N-1]
		\end{cases}\\
		&= 2\left(\tau_{n,i}\delta_i - \delta_i^2 + \tau_{n,i}\right)\eta_i\\
		&= 2\tau_{n,i}(\delta_i-1)\eta_i.
	\end{align*}
	Rearranging this identity, we have that for all $i\in[n,N-1]$, that $\delta_{i+1}\eta_{i+1} = (\delta_i - 1)\eta_i$.
	
	We prove the first statement inductively.
	For $i = n$, we have that
	\begin{align*}
		\frac{1}{2\tau_{n,n}(\delta_n - 1)} 
		&= \frac{1}{2(\tau_{n-1/2}\delta_n + \tau_{n,n})}\leq \eta_n = \frac{1}{2\tau_{n-1/2}\delta_n}.
	\end{align*}
	
	Now, suppose $i\in[n+1,N]$, then
	\begin{align*}
		\frac{1}{2\tau_{n,i}(\delta_{i}-1)}
		&\leq \frac{1}{2\tau_{n,i-1}\delta_i}\\
		&\leq \eta_{i} = \frac{(\delta_{i-1}-1)\eta_{i-1}}{\delta_i}\\
		&\leq \frac{(\delta_{i-1}-1)}{2\tau_{n-1/2}\delta_{i-1}\delta_i}\\
		&\leq \frac{1}{2\tau_{n-1/2}\delta_i}.
	\end{align*}
	Here, the first line follows from \eqref{eq:taupsi_identity}, the second line follows by definition of $\eta_i$, and the third line follows by induction.

	For the second claim, suppose $i\in[n,N-1]$.
	Then,
	\begin{align*}
		\tau_{n,i}(\delta_i-1)\eta_i&=\tau_{n,i}\delta_{i+1}\eta_{i+1} \leq \tau_{n,i+1}(\delta_{i+1}-1)\eta_{i+1},
	\end{align*}
	where the inequality follows from \eqref{eq:taupsi_identity}.
	
	For the third claim,
	first suppose $i\in[n,N-1]$ and $j = i+1$. Then, 
	$\delta_j\eta_j = (\delta_i - 1)\eta_i$ so that the claim holds. Now, suppose $n\leq i < j \leq N$. Then, we can chain together this inequality to get
	\begin{align*}
		\delta_j \eta_j &\leq (\delta_{j-1} - 1)\eta_{j-1}\\
		&\leq \delta_{j-1}\eta_{j-1}\\
		&\leq (\delta_{j-2} - 1)\eta_{j-2}\\
		&\leq \dots\\
		&\leq (\delta_{i} - 1)\eta_i.\qedhere
	\end{align*}
\end{proof}

\begin{proof}[Proof of \cref{lem:crossterm2}]
	We handle the $i=n$ and $i>n$ cases separately.
	First, suppose $j< i = n$. Then,
	\begin{align*}
		\ip{g_j, x_n^+ - x_j^+}&= \ip{g_j, x_n^+ - z_{n+1/2}} + \ip{g_j, z_{n+1/2}-x_j^+}\\
		&= \frac{\tau_{n-1/2}}{\tau_{n,n}}\ip{g_j, x_{n-1/2}^+ - z_{n+1/2}} + \ip{g_j, z_{n+1/2}-x_j^+},
	\end{align*}
	where the second line uses the fact that $g_n$ is orthogonal to $g_j$ and the recursive definition of $x_n$.
	
	We will now assume $i>n$ for the remainder of the proof. We expand
	\begin{equation*}
		\ip{g_j, x_i^+ - x_j^+} = 
		\ip{g_j, x_i^+ - z_i} + \ip{g_j, z_i - x_j^+},
	\end{equation*}
	By definition, $z_i = z_{n+1/2}- \sum_{\ell=n}^{i - 1}\frac{\delta_\ell}{L} g_\ell$, thus the second term is
	\begin{equation*}
		\ip{g_j, z_i - x_j^+} = \begin{cases}
			-(\delta_j-1) \eta_j L\Delta&\text{if }j \geq n\\
			\ip{g_j, z_{n+1/2} - x_j^+}&\text{if } j<n.
		\end{cases}
	\end{equation*}
	For the first term, we use the inductive definition of $x_i$, $z_i$, and $\tau_{n,i-1}$ to write
	\begin{align*}
		\ip{g_j, x_i^+ - z_i} 
		&=\ip{g_j, x_i - z_i}\\
		&=\frac{\tau_{n,i-1}}{\tau_{n,i}}\ip{g_j, x_{i-1}^+ - z_i}
		\\
		&= \begin{cases}
			\frac{\tau_{n,n}}{\tau_{n,n+1}}\ip{g_j, x_n^+ - z_{n+1/2} + \frac{\delta_{n}}{L}g_{i-1}}&\text{if }i=n+1\\
			\frac{\tau_{n,i-1}}{\tau_{n,i}}\ip{g_j, x_{i-1}^+ - z_{i-1} + \frac{\delta_{i-1}}{L}g_{i-1}}&\text{if }i\geq n+2.
		\end{cases}
	\end{align*}
	Note that in both cases, if $j<i-1$ then $\ip{g_j, g_{i-1}}=0$. Thus, we may unroll the recursion to get 
	\begin{equation*}
		\ip{g_j, x_i^+ - z_i} = \begin{cases}
			\frac{\tau_{n,j}}{\tau_{n,i}}\ip{g_j, x_{j}^+ - z_{j} + \frac{\delta_{j}}{L}g_{j}}
			&\text{if }j>n\\
			\frac{\tau_{n,n}}{\tau_{n,i}}\ip{g_n, x_{n}^+ - z_{n+1/2} + \frac{\delta_{n}}{L}g_{n}}&\text{if }j = n\\
			\frac{\tau_{n-1/2}}{\tau_{n,i}}\ip{g_j, x_{n-1/2}^+ - z_{n+1/2}}
			& \text{if }j < n.
		\end{cases}
	\end{equation*}
	Next, noting that for $j>n$, $g_j$ is orthogonal to $x_j - z_j$ and for $j= n$, $g_j$ is orthogonal to $x_j - z_{n+1/2}$, we simplify
	\begin{equation*}
		\ip{g_j, x_i^+ - z_i} = \begin{cases}
			\frac{\tau_{n,j}}{\tau_{n,i}}(\delta_j - 1)L\eta_j\Delta
			&\text{if }j\geq n\\
			\frac{\tau_{n-1/2}}{\tau_{n,i}}\ip{g_j, x_{n-1/2}^+ - z_{n+1/2}}
			& \text{if }j < n.
		\end{cases}
	\end{equation*}
	Finally, we combine the summands to get that for $i>n$,
	\begin{equation*}
		\ip{g_j, x_i^+ - x_j^+} = 
		\begin{cases}
			\left(\frac{\tau_{n,j}}{\tau_{n,i}}-1\right)(\delta_j - 1)L\eta_j\Delta
			&\text{if }j\geq n\\
			\frac{\tau_{n-1/2}}{\tau_{n,i}}\ip{g_j, x_{n-1/2}^+ - z_{n+1/2}} +  \ip{g_j, z_{n+1/2} - x_j^+}
			& \text{if }j < n.
		\end{cases}\qedhere
	\end{equation*}
\end{proof}

\begin{proof}
	[Proof of \cref{lem:subopt_value}]
	
	We will require the identity
	\begin{equation*}
		\delta_N^2 = \tau_{n,N},
	\end{equation*}
	which follows from the definition of $\delta$ in~\eqref{eq:delta} and $\tau_{n,i}$ in \eqref{eq:tau_ni}.
	
	Now, we compute
	\begin{align*}
		&f_N - f_\star - \frac{L}{2\tau_{n,N}}\norm{x_0-z_{N+1}}^2\\
		&= \frac{L\Delta}{2}\left((2\delta_N-1)\eta_N - \frac{1}{\tau_{n,N}}
		\left(1  + \sum_{j=n}^{N}
		\delta_j^2 \eta_j \right)\right)\\
		&= \frac{L\Delta}{2}\left(2(\delta_N-1)\eta_N - \frac{1}{\tau_{n,N}}
		\left(1  + \sum_{j=n}^{N-1}
		\delta_j^2 \eta_j \right)\right)\\
		&= \frac{L\Delta}{2}\left(\frac{\delta_N-1}{\tau_{n,N-1}\delta_N} - \frac{1}{\tau_{n,N}}
		\right) \left(1  + \sum_{j=n}^{N-1}
		\delta_j^2 \eta_j \right)\\
		&= \frac{L\Delta}{2}\left(\frac{
			\delta_N(\tau_{n,N-1}+\delta_N) - \tau_{n,N} - \tau_{n,N-1}\delta_N  
		}{\tau_{n,N-1}\tau_{n,N}\delta_N}
		\right) \left(1  + \sum_{j=n}^{i-1}
		\delta_j^2 \eta_j \right).
	\end{align*}
	Here, the second line substitutes definitions from our construction, the third line notes that $\delta_N^2\eta_N = \tau_{n,N}\eta_N$, the fourth line substitutes the definition of $\eta_N$, and the fifth line is algebra.
	
	It remains to observe that the numerator in the second term is identically zero:
	\begin{equation*}
		\delta_N(\tau_{n,N-1}+\delta_N) - \tau_{n,N} - \tau_{n,N-1}\delta_N   =  \delta_N^2 - \tau_{n,N}= 0.\qedhere
	\end{equation*}
	\end{proof}
\end{document}